\documentclass[a4paper,11pt]{article}
\usepackage{amsthm}
\usepackage{amsmath}
\usepackage{amsfonts}
\usepackage{dsfont}
\usepackage{latexsym}
\usepackage{amssymb}
\usepackage[latin1]{inputenc}

\newtheorem{fed}{Definition}[subsection]
\newtheorem{teo}[fed]{Theorem}
\newtheorem*{teo*}{Theorem}
\newtheorem{lem}[fed]{Lemma}
\newtheorem{cor}[fed]{Corollary}
\newtheorem{pro}[fed]{Proposition}

\theoremstyle{definition}
\newtheorem{rem}[fed]{Remark}
\newtheorem{conj}[fed]{Conjecture}
\newtheorem{Not}[fed]{Notations}

\newtheorem{exa}[fed]{Example}

\newcommand{\mat}{\mathcal{M}_n (\C) }
\newcommand{\matm}{\mathcal{M}_m (\C) }

\newcommand{\matreal}{\mathcal{M}_n (\R) }
\newcommand{\matmreal}{\mathcal{M}_m (\R) }

\def\barr{\begin{array}}
\def\earr{\end{array}}

\newcommand{\diag}[1]{\hbox{\rm diag}\left( #1\right)}

\newcommand{\CC}{\mathbb{C}}
\newcommand{\RR}{\mathbb{R}}

\newcommand{\F}{\mathbb{F}}

\def\prop{Proposition \ref}
\def\teor{Theorem \ref}

\def\lema{Lemma \ref}

\def\suml{\sum\limits}

\def\R{\mathbb{R}}
\def\C{\mathbb{C}}

\def\noi{\noindent}
\def\bma{\left[\begin{array}}
\def\ema{\end{array}\right]}
\def\ben{\begin{enumerate}}
\def\een{\end{enumerate}}
\newcommand{\IN}[1]{\mathbb {I} _{#1}}

\DeclareMathOperator{\leqp}{\underset {ij}{\leqslant}}
\DeclareMathOperator{\leqpi}{\leqslant}
\DeclareMathOperator{\geqp}{\geqslant}
\DeclareMathOperator{\leqpp}{\leqslant}

\usepackage[dvips]{color}

\def\ds{\displaystyle}
\def\noi{\noindent}
\definecolor{color}{RGB}{140,140,140}
\def\QED{\hfill $\square$}
\def\EOE{\hfill $\triangle$}
\def\QEDP{\tag*{\QED}}
\def\EOEP{\tag*{\EOE}}
\def\bdem{\begin{proof}}

\def\la{\lambda}

\def\pa{{\bf a}}
\newcommand{\matinv}{\mathcal{G}\textit{l}\,(n)}

\newcommand{\matrec}[1]{\mathcal{M}_{#1} (\mathbb{C})}

\def\N{\mathcal{N}}

\def\cB{\mathcal{B}}

\def\cF{\mathcal{F}}

\def\cM{\mathcal{M}}

\def\cP{\mathcal{P}}

\def\cS{\mathcal{S}}

\def\cU{\mathcal{U}}
\def\cW{\mathcal{W}}

\def\uno{\mathds{1}}
\def\E{\mathds{E}}

\newcommand{\matu}{\mathcal{U}(n)}

 \DeclareMathOperator{\tr}{tr}

     \DeclareMathOperator{\sop}{supp}
     
\newcommand{\ba}{\mathbf a }

\def\pausa{\medskip\noindent}

\newcommand{\pint}[2]{\displaystyle \left \langle #1 \,, \, #2  \right\rangle}

\newcommand{\hil}{\mathcal{H}}

\newcommand{\cene}{\mathbb{C}^n}

\newcommand{\matsa}{\mat_{sa}}
\newcommand{\matah}{\mat_{ah}}

\newcommand{\matsauno}{\cM_n^1(\C)_{sa}}
\newcommand{\matsao}{\cM_n^0(\C)_{sa}}
\newcommand{\matpos}{\mathcal{M}_n(\C)^+}

\newcommand{\convk}{\xrightarrow[k\rightarrow\infty]{}}

\newcommand{\frs}{S_{\cW_w}}

\def\ssec{\cW = \{ W_i\}_{i\in \IM}}

\def\FS {\cW_w\,}
\def\sfram{\cW_w  = (w_i\, ,\, W_i)_{i\in \IM}}
\def\msfram{\cW_w  = (w_i\, ,\, W_i)_{i\in \IN{m}} \,}

\def\inv{^{-1}}

\def\beq{\begin{equation}}
\def\eeq{\end{equation}}
\def\N{\mathbb{N}}
\def\F{\mathcal{F}}

\def\cB{\mathcal{B}}

\def\cU{\mathcal{U}}

\def\cP{\mathcal{P}}

\def\ese{\mathcal{M}}

\def\eme{\mathcal{M}}
\def\ene{\mathcal{N}}
\def\cV{\mathcal{V}}
\def\cW{\mathcal{W}}

\def\inc{\subseteq}

\def\orto{^\perp}
\def\inc{\subseteq}
\def\sii{ if and only if }
\def\inv{^{-1}}
\def\*A{\#\sb A}

\def\H{{\cal H}}

\def\cP{{\cal P}}

\def\cS{{\cal S}}

\def\cM{{\cal M}}

\def\rai{^{\frac {1}{\,2}}}
\def\mrai{^{\frac{-1}{\ 2}}}

\def\api{\langle}
\def\cpi{\rangle}

\def\noi{\noindent}

\def\bm{\left(\begin{array}}
\def\em{\end{array}\right)}
\def\ben{\begin{enumerate}}
\def\een{\end{enumerate}}
\def\barr{\begin{array}}
\def\earr{\end{array}}

\def\inv{^{-1}}

\def\H{{\cal H}}

\def\la{\lambda}

\def\In{\mathbb {I} _n}
\def\IM{\mathbb {I} _m}

\def\lh{{L(\H)}}

\def\lh+{{\lh^+}}

\def\noi{\noindent}

\def\inc{\subseteq}
\def\bm{\left(\begin{array}}
\def\em{\end{array}\right)}

\newcommand{\peso}[1]{ \quad \text{ #1 } \quad }
\newcommand{\argmin}{\operatornamewithlimits{arg\,min}}

\newcommand{\trivial}{\{0\}}
\newcommand{\gen}[1]{\mbox{span}\left\{#1\right\}}

\def\bsim{ \R_*^m }

\newcommand{\FP}{\text{FFP}\, }
\newcommand{\fp}{\text{FP}\, }

\date{}

\begin{document}
\title{\bf The structure of minimizers of the frame potential \\ 
on fusion frames
\footnote{Keywords: Fusion frames, frame potential, majorization, Hadamard product.}
\footnote{Mathematics subject classification (2000): 42C15, 15A60.}}
\author{Pedro G. Massey, Mariano A. Ruiz  and Demetrio Stojanoff\thanks{Partially supported by CONICET 
(PIP 5272/05) and  Universidad de La PLata (UNLP 11 X472).} }

\maketitle

\begin{abstract}
In this paper we study the fusion frame potential, that is a generalization of the Benedetto-Fickus 
(vectorial) frame potential to the finite-dimensional fusion frame setting. 
The structure of local and global minimizers of this potential is studied, when we restrict 
the frame potential  to suitable sets of  fusion frames. These minimizers are related to tight 
fusion frames as in the classical vector frame case. Still, tight fusion frames 
are not as frequent as tight frames; indeed we show that there are choices of 
parameters involved in fusion frames for which no tight fusion frame can exist. 
Thus, we exhibit necessary and sufficient conditions for the existence of tight 
fusion frames with prescribed parameters, involving the so-called Horn-Klyachko's 
compatibility inequalities.
The second part of the work is devoted to the study of the minimization of the 
fusion frame potential on a fixed sequence of subspaces, varying the sequence of weights. 
We related this problem to the index of the Hadamard product by positive 
matrices and use it to give different characterizations of these minima. 
\end{abstract}

\tableofcontents

\section{Introduction}
Fusion frames were introduced by P. G. Casazza and G. Kutyniok  under the name of ``frame of subspaces'' in \cite{[CasKu]}. They are a generalization of the usual frames of vectors for a Hilbert space $\hil$; indeed frames of vectors can be treated as ``one dimensional fusion frames''.  
During the last years, the theory of  fusion frames has been a fast-growing area. Several applications of fusion frames have been studied, for example, 
sensor networks \cite{[CasKuLiRo]}, neurology \cite{RJ}, 
coding theory \cite{Bod}, \cite {Pau}  , \cite{GK}, among others. 
Specifically, applications which require distributed processing can be well 
described and studied using fusion frames. We refer the reader to the work by 
P.G. Casazza, G. Kutyniok and S. Li \cite{[CasKuLi]} and the references therein,  
for a detailed treatment of the fusion frame theory. Further developments 
can be found in  \cite{CF}, \cite {[CasKu2]} and \cite{RS}. 

In finite dimension, a fusion frame is a sequence of subspaces of $\mathbb{F}^n$ 
($\mathbb{F}=\CC$ or $\RR$) together with a set of positive weights such that the  
weighted  sum of the orthogonal projections to these subspaces (called {\it fusion 
frame operator}) is a positive invertible operator (see Definition  \ref{def: fusionframes}).  
As in the case of vector frames, it is usually desired this invertible operator to 
be a multiple of the identity. In this case, the fame is called a tight fusion frame, also noted TFF. 
However, tight fusion frames might not exist for a fixed choice of dimensions 
$\mathbf d =(d_i)_{i\in \IM}$ of the  subspaces, 
for any sequence of weights $w=(w_i)_{i\in \IM}\,$ (see the discussion following \prop{restric}).

Inspired on the work in classical frame theory (see \cite{[BF],MR}), 
we study a convex functional on fusion frame operators (the FF potential or FFP),
also studied by Casazza and Fickus in \cite{CF}, which generalizes 
the Benedetto-Fickus frame potential. In this paper we 
analyze its local and global minima.  
This is motivated by the fact that local (global) minimizers characterize unit norm tight 
vector frames  
(see \cite{[BF],Bod,Pau,casazza2,MR}). 
Since the FF potential can be seen as a ``measure of orthogonality'' of the frame vectors, 
it provides also an interesting geometrical description of fusion frames. 
These considerations motivate the  study this type of minimizations in the fusion frame context. 
A related study can be found in \cite{CF}, a work that the authors 
became aware in an 
advanced stage of writing this paper.

The main tool used in \cite{MR} for these  problems for vector frames, 
namely majorization of matrices, 
can be replaced in the context of fusion frames by the  theory developed by Horn  
and Klyachko in order to have a spectral characterization of hermitian matrices 
which are the  sum of a set of hermitian matrices. 
For example, this approach provides necessary and sufficient conditions of existence of TFF's, summarized in a family of inequalities. This technique, although seems to be rather impractical due to the 
complex conditions involved, becomes 
an useful tool in the study of the spectral structure of the FF potential minimizers.

We first consider the problem of existence of TFF's. We show some dimensional restrictions
on the subspaces regarding this problem, and then give equivalent conditions for the 
case of fixed dimensions and weights. We refer to \cite{CF} for further 
developments in this direction.

The rest of the paper deals with the minimization of the FF potential on some sets of Fusion 
Bessel sequences (i.e. sets of projections and weights whose fusion frame operator is not 
necessary invertible). Mainly, we work on Bessel sequences with fusion frame operator 
of trace one. This is a natural restriction in order to avoid scalar multiplications, 
and it allows an interpretation of the FF potential as a measure of how far is the 
fusion frame operator to the suitable multiple of the identity corresponding to 
the (possibly non-existing) tight fusion frames with trace one. A detailed discussion
of this approach can be found in subsection \ref{2.2}.

The minimization of the FFP is done in three different settings: first, by fixing the weights $w$ 
and the dimensions $\mathbf d$ of the subspaces. Then, fixing only the dimensions $\mathbf d$. 
Finally, we consider a fixed sequence of subspaces $(W_i)_{i\in \IM} $, and optimize over the set 
of admissible weights. In the three cases, the minimization is made under the previously mentioned ``trace one" restriction.  
    
For the first problem, a geometrical approach similar to that done in \cite{MR} allows us 
to obtain a characterization of local minimizers of the FF potential: they are orthogonal sums 
of tight frames on each eigenspace of the frame operator.  
Then, using Horn and Klyachko techniques, we 
prove that all minimizers (even those which are local minimizers) 
have the same eigenvalues, with the same multiplicities. 
Similar results are obtained for the second problem (fixing only the dimensions of the subspaces).

The last section of the paper is devoted to the study of the optimization of the fusion frame potential of fusion frames obtained from a fixed sequence of subspaces in $\CC^n$ which generate the whole space. 
Since every sequence of weights makes it a fusion frame, we seek for the 
best choice of weights, meaning those which minimize (globally) the FF potential. 
Then, we establish  a connection between this optimization problem and 
the {\it Hadamard indexes} studied in \cite{[CS]}  
(which involve the Hadamard or entry-wise product of matrices) of a 
kind of Gramm matrix associated to  the fixed subspaces. 
Using these tools, we get 
a characterization of the set of optimal weights, and a way to compute them under 
some reasonably assumptions on the initial sequence of subspaces. This analysis 
seems to be new even  for the case of vector frames.

However, as it is shown by an example, the  minimizers  could ``erase'' some of 
the initial subspaces (i.e. the set of optimal weights could have zeros).  
Moreover, it is possible to obtain a minimizer which is  a Bessel sequence of subspaces 
which is not generating, a phenomenon  which does not happen in the 
previous settings. Motivated by this problem, we study the geometry of the 
set of all weights $w$ which minimize the FF potential, and in particular their 
possible supports (namely, those sub-indexes $i$ such that $w_i>0$). 
At the end of the section, we present some examples which illustrate
these type of anomalies. 

The paper is organized as follows: Section 2 contains preliminary definitions on 
fusion frames and the basic notation used throughout the paper.  This section 
also contains  a  brief exposition of Horn-Klyachko's compatibility inequalities. 
Section 3 is devoted to the study of minimizers of the FF potential, restricted 
to the sets of fusion frames detailed before. 
In Section 4 we analyze the problem of minimize the FF potential 
for a fixed sequence of subspaces, varying the weights. 
The paper ends with an appendix containing definitions and several 
results concerning Hadamard indexes of positive matrices, which are 
related to the contents of Section 4.

We wish to acknowledge to P. Casazza and M. Fickus for letting us know
about their excellent work \cite{CF}, which is closely related 
with the present paper.

\def\subim{_{i\in \IN{m}}\,}
\def\subir{_{i\in \IN{r}}\,} 

\section{Preliminaries and Notations.}
In this paper $\mat$ denotes the algebra of complex $n\times n$ matrices, 
$\matinv$ the group of all invertible elements of $\mat$, $\matu$ the group 
of unitary matrices,  $\matsa$ (resp. $\matah$) denotes the real subspace of 
hermitian (resp. anti-hermitian) matrices, $\matpos$ the set of positive semidefinite
matrices, and $\matinv^+ = \matpos \cap \matinv$. Given $T \in \mat$, $R(T)$ denotes the
image of $T$, $N(T)$ the null space of $T$, 
$\sigma (T)$ the spectrum of $T$, $\tr T$ the trace of $T$,
and rk $T $  the rank of $T$.

\pausa 
Given $m \in \N$ we denote by $\IM = \{1, \dots , m\}$ and 
$\uno = \uno_m  \in \R^m$ denotes the vector with all its entries equal to $1$. 
If $v \in \R^n$, we denote by 
$\mbox{\rm diag}(v) \in \mat$ the diagonal matrix with $v$ in its diagonal, 
and by  
$v^\downarrow\in \RR^n$ the vector
obtained by re-arrangement of the coordinates of $v$ in
non-increasing order. If $T\in \matsa\,$, we denote by $\la(T) \in \R^n$ the vector 
of eigenvalues of $T$, counted with multiplicities, in such 
a way that $\la (T) = \la(T)^\downarrow$. 

\pausa
Given a subspace $W\inc \cene$, we denote by $P_W \in \matpos$ the orthogonal 
projection onto $W$, i.e. $R(P_W) = W$ and $N(P_W) = W^\perp$. 
For vectors on $\cene$ we shall use the euclidean norm, but for matrices $T\in \mat$, we shall use both the 
spectral norm $\|T\| = \|T\|_{sp}= \max\limits_{\|x\|=1}\|Tx\|$, 
 and the Frobenius norm $\|T\|_{_2} = (\tr \, T^*T )\rai = 
\big( \, \suml_{i,j \in \In } |T_{ij}|^2 \, \big)\rai$. This norm is induced by the inner product
$\api A,\ B\cpi=  \tr \, B^*A \,$,  for $A, B \in \mat$.

\subsection{Frames of subspaces, or fusion frames for $\mathbb C^n$}

We begin by defining the basic notions of fusion frame theory in the finite dimensional 
context. For an introduction to fusion frames for general Hilbert spaces, 
see \cite{[CasKu]},  \cite{[CasKuLi]} or \cite{RS} . 
Briefly, a fusion frame for $\cene$ is a generating sequence of subspaces, equipped 
with weights assigned to each subspace. Nevertheless, we prefer to give the ``frame style" definition, 
which adjusts better to our purposes.

\begin{fed}\label{def: fusionframes} \rm
Let  $\ssec$ be closed subspaces of $\hil \cong\C^n$, and
$w=  \{w_i\}_{i\in \,\IM} \in \R_{>0}^m$.
The sequence $\sfram$ is a {\it fusion frame  (FF) for $\hil$},
if there exist $A,B>0$  such that
\begin{equation}\label{frame de sub}
A\|f\|^2\leq \sum_{i\in \IM}w_i^2 \, \|P_{W_i}f\|^2 \leq B\|f\|^2 \peso {for every $f\in \hil$
\ .}
\end{equation}
If only the right-hand side inequality in \eqref{frame de sub} holds,  then we say that $\cW_w$ is a {\it Bessel sequence of subspaces} (BSS) for $\hil$. The {\bf frame operator} of $\FS $ is defined by 
the formula 
\beq\label{SWw}
\frs = \sum_{i\in \IM} w_i^2 \, P_{W_i} \in \matpos \ .
\eeq
Observe that 
$\FS$ is a FF \sii $\frs  \in \matinv^+$ and, in this case $A \, I_n \le \frs  
\le B \, I_n \,$. 
We say that 
$\cW_w$ is  a {\it tight} FF (TFF) if $A = B \,$, in other words, if $\frs = A I_n\, $. 
\EOE
\end{fed}

\begin{rem}\label{de FS a vec}
Let $\sfram$ be a BSS for $\cene$. 
For each $i\in \IM\,$, we can take an o.n.b. $\cB_i = \{e_j^{(i)}\}_{j\in J_i}$ of $W_i\,$. 
Hence, for every $f \in \hil$, we have that  
$$ 
P_{W_i}f=\sum_{j\in J_i}\langle f,e_j^{(i)}\rangle\,e_j^{(i)} \, , \ \mbox{ for }  \  i\in \IM 
 \  \implies  \ 
 S_{\cW_w} f =\sum_{i\in \IM}w_i^2 P_{W_i} f=
 \sum_{i\in \IM}\sum_{j\in J_i} \langle f,w_i\,e_j^{(i)} \rangle\, w_i \,e_j^{(i)} \ .
$$
Therefore, $\cW_w$ induces a vector Bessel sequence $\F=\{w_i\, e_j^{(i)}:\ i\in \IM \, ,\ j\in J_i\}$ 
which has a very useful property: Its frame operator $S_\F = S_{\cW_w}\,$. 
\EOE
\end{rem}
\def\d{\mathbf d}

\subsection{Sets of fusion frames and the FF-potential}\label{2.2}
We shall establish
several notations regarding sets of FF's and BSS's :

\begin{Not} \label{Bmn}
Fix $n,\, m \in \N$ and consider  a Hilbert space   $\hil \cong \cene$. 
\ben
\item We shall denote by $\cS_{m\, , n} $ the set of all FF's of the form 
$\FS = (w, \cW) = (w_i,W_i)_{i\in\IM}$, where $w \in \R_{>0}^m\,$  and $\cW$ a generating 
sequence of subspaces of $\hil$. 
\item Given a sequence $\d \in \N^m$ such that $\tr \d = \suml_{i \in \IM} d_i \ge n$, we denote by 
\beq\label{con d}
\cS_{m\, , n}(\d)  = \big\{ \FS \in \cS_{m\, , n}  \ :  \ \dim W_i = d_i 
\peso{for every} i \in \IM \big\} \ , 
\eeq
Similarly, we denote by  $\cB_{m\, , n}(\d) $ 
 the set of BSS's with the same dimensional restrictions. 
\item Given  $v \in \R_{>0}^m\,$, we denote by
\beq\label{conjuntos}
\cB_{m, \, n}(\mathbf d,v)=\{\cW_w \in \cB_{m\, , n}(\d ) : w = v \}
\peso{and} 
 \cS_{m, \, n}(\mathbf d,v)= \cB_{m, \, n}(\mathbf d,v) \cap \cS_{m\, , n}(\d) \ ,
\eeq
the subsets of $\cB_{m\, , n}(\d) $ and $\cS_{m\, , n}(\d)$ with a fixed sequence of weights $v$.
 
\item Finally, we denote by 
\beq\label{con tr1}
\barr{rl}
\cB_{m\, , n}^1(\d)  
= \big\{ \FS \in \cB_{m\, , n}(\d) : \tr S_{\cW_w}=\suml_{i \in \IM}\,  w_i^2\,d_i=1  \big\} 
&  \ .
\earr
\eeq 
and $\cS_{m\, , n}^1(\d) =\cB_{m\, , n}^1(\d) \cap \cS_{m\, , n}(\d)\,$. 
\item 
We say that a pair 
$(\mathbf d,w)\in \N^m\times\RR_{>0}^m\,$ 
is {\bf normalized} if $\tr \d \ge n$ and $\suml_{i \in \IM}\,  w_i^2\,d_i=1$. Observe that 
$(\mathbf d,w)$ is normalized \sii $\cB_{m, \, n}(\mathbf d,w) \inc \cB_{m\, , n}^1(\d)  $. 
\EOE
\een
\end{Not}

\pausa
The following definition is suggested by the classical Benedetto-Fickus potential, whose value in a vector frame $\cF$ can be calculated as FP$(\F ) = \tr S_\cF  ^{\, 2}\,$. 

\begin{fed} \rm 
Given a BSS $\cW_w=(w_i,W_i)_{i \in \IM}\, $, 
the Benedetto-Fickus {\it fusion frame potential} (FF-potential) of $\cW_w$ is given by:
\begin{equation}\label{defi del fusion pot} 
\FP( \cW_w)=\suml_{i,j=1}^m w_i^2w_j^2\tr(P_{W_i}P_{W_j}) = \tr \, S_{\FS} ^{\, 2}\ .
\end{equation}
Notice that in this case $\FP(\cW_w)=\text{FP}(\cF)$, for any vector Bessel sequence $\F$ 
obtained from $\cW_w$ as in Remark \ref{de FS a vec}. We define also the following matrix:
\begin{equation}\label{defi qpot}
\mathbf P_q(\cW_w)=\suml_{i,\,j=1}^m w_i^2w_j^2\,|P_{W_i}\,P_{W_j}|^2=
\suml_{i,\,j=1}^m w_i^2w_j^2\,P_{W_j}\,P_{W_i}\,P_{W_j} \in \matpos\ .
\end{equation} 
The matrix $\mathbf P_q(\cdot)$ is related to the so-called q-potential
 \cite{P} defined in the more general context of reconstruction systems. Notice that the Benedetto-Fickus fusion frame potential can be computed in terms of $\mathbf P_q(\cW_w)$, since $\FP(\cW_w)=\tr \mathbf P_q(\cW_w)\,$.
 \EOE
\end{fed}

\pausa
The scope of this paper is to study minimizers of the FF-potential. In order to avoid 
scalar multiplications (note that $\FP ( \cW_{t \cdot \, w} ) = t^4 \, \FP ( \FS )\,$) we shall restrict
ourselves to minimize the FF-potential 
on subsets of $\cS_{m\, , n}^1(\d)$ or $\cB_{m\, , n}^1(\d)$, for $\d$ and $m$
fixed. In other words, we shall minimize the FF-potential for those 
frames $\FS$ such that $\tr S_{\FS}  = 1$. 

\pausa
This specific restriction is justified because, if there exist tight FF's in 
$\cS_{m\, , n}(\d)$, then their FF-potential and frame bounds 
are determined exactly by the trace 
of their frame operators. Namely, if $\FS\in \cS_{m\, , n}(\d) $ is tight, and 
$\tr S_{\FS}  = a$, then $S_{\FS}  = \frac an \, I_n \,$ and 
$\FP( \cW_w)  =  \frac {\  a^2}{n} \, $.

\pausa
Even in the case that there are no TFF's in $\cS_{m\, , n}(\d)$, this restriction 
seems to be quite natural. Indeed, for BSS's with fixed trace,  the FF-potential 
can be seen as a measure of the (Frobenius) distance of their frame operators to a 
{\bf fixed} multiple of the identity: 

\begin{pro}\label{dist a la ident} 
Let $ \FS \in \cS_{m, \, n}^1(\d)$. Then 
\[
\big\|\frac{1}{\,n} \ I_n - S_{\FS}\big\|_{_2}^2= \tr S_{\FS}^{\, 2} \ - \ \frac{1}{\,n}
 \ =  \ \FP( \cW_w) \ - \ \frac{1}{\,n} \ \ .
\]
\end{pro}
\proof
Since $\tr S_{ \FS}=1$, a direct computation shows that 
\beq
\big\|\frac{1}{\,n} \ I_n - S_{ \FS}\big\|_{_2}^2=\tr \left(\frac{1}{n^2} \ I_n - \frac{2}{n}\ S_{ \FS}
+ S_{ \FS} ^{\, 2} \right)  = \tr S_{\FS}^{\, 2} \ - \ \frac{1}{\,n}  \ \ . 
\QEDP
\eeq

\pausa
The last result shows that if there exist tight FF's in 
$\cS_{m ,\,  n}^1(\d)$, then they are the unique global minimizers of the FF-potential 
on $\cS_{m ,\,  n}^1(\d)$. But in the case that there are no TFF's in $\cS^1_{m ,\,  n}(\d)$, 
the minimization of the FF-potential becomes more interesting: 
it provides the elements of  $\cS_{m ,\,  n}^1(\d)$ that can be expected to have 
the best properties. 

\pausa
In this paper we deal mostly with these type of minimizations under two different 
further restrictions: we work in the set 
$\cB_{m ,\,  n}(\d , w)\inc \cB_{m ,\,  n}^1(\d) $ for a 
fixed normalized pair $(\d \, , \, w)$, 
or we fix a generating  sequence $\cW$ of subspaces, and minimize 
the FF-potential over all sequences $w\in \R_{\geq 0} ^m$ 
such that $\FS \in \cB_{m ,\,  n}^1(\d)$.

\subsection{Klyachko-Fulton approach}

Recall that given $x\in \RR^n$, we denote by $x^\downarrow\in \RR^n$ the vector
obtained by re-arrangement of the coordinates of $x$ in
non-increasing order. Given $x,\,y\in \RR^n$ we say that $x$ is
{\it submajorized} by $y$, and write $x\prec_w y$, if
$\suml_{i=1}^k x^\downarrow _i\leq \suml_{i=1}^k y^\downarrow _i \,$ 
for every $k\in \In \,$.  
If we further have that
$\tr(x):=\suml_{i=1}^nx_i=\suml_{i=1}^n y_i$ then we say that $x$ is
majorized by $y$, and write $x\prec y$.

\begin{exa}\label{ejem}
As an elementary example, that we shall use in what
follows, let $x\in \RR_{\geq 0}^n$ and $0\leq a\leq \tr(x)\leq b$.
The reader can easily verify that 
$\frac{a}{n}\,\uno_n \prec_w x\prec_w b\, e_1 \, $. \EOE
\end{exa}

\pausa
(Sub)majorization between vectors is extended by T. Ando in \cite{Ando} 
to (sub)majorization
between self-adjoint matrices as follows : 
given $A,\,B\in
\matsa\,$, we say that $A$ is submajorized by $B$, and
write $A\prec _w B$, if $\lambda(A)\prec_w \lambda(B)$. If we
further have that $\tr(A)=\tr(B)$ then we say that $A$ is majorized
by $B$ and write $A\prec B$.

\pausa
Although simple, submajorization plays a central role in
optimization problems with respect to convex functionals and
unitarily invariant norms, as the following result shows (for a
detailed account in majorization see Bhatia's book \cite{Ba}).

\begin{teo}\label{props submayo y Klyachko}\rm
Let $A,\,B\in \matrec{n}^{sa}$. Then, the following statements are
equivalent:
\begin{enumerate}
\item $A\prec_w B$.
\item For every unitarily invariant norm $\|\cdot\|$ in $\mat$ we have
$\|A\|\leq \|B\|$.
\item For every increasing convex function $f:\RR\rightarrow \RR$ we
have $\tr f(A) \leq \tr f(B) $.
\end{enumerate}Moreover, if $A\prec_w B$ and there exists an
increasing  strictly convex function $f:\RR\rightarrow \RR$ such
that $\tr f(A) = \tr f(B) $ then there exists $U\in \cU(n)$ such that
$A=U^*BU$. \QED
\end{teo}

\pausa
In what follows we describe the basic facts about the spectral
characterization of the sums of hermitian matrices obtained by
Klyachko \cite{Klya} and Fulton \cite{Ful}.
Let 
$$
\mathcal K_r^n=\big\{(j_1,\ldots,j_r)\in \In^r :\  j_1<j_2\ldots<j_r \big\} \ . 
$$
For $J=(j_1,\ldots,j_r)\in \mathcal K_r^n$, define the
associated partition
$$
\lambda(J)=(j_r-r,\ldots,j_1-1) .
$$
Denote by $LR_r^{\,n}(m)$ the set of
$(m+1)$-tuples $(J_0,\ldots,J_m)\in (\mathcal K_r^n)^{m+1}$, such
that the Littlewood-Richardson coefficient of the associated
partitions $\lambda(J_0),\ldots,\lambda(J_m)$ is positive, i.e. one
can generate the Young diagram of $\lambda(J_0)$ from those of
$\lambda(J_1),\ldots,$ $\lambda(J_m)$ according to the
Littlewood-Richardson rule (see \cite{Ful0}). With these notations
and terminologies we have
\begin{teo}[\cite{Klya,Ful}]\label{teoK}\rm
Let $\lambda_i=\lambda_i^\downarrow=(\lambda^{(i)}_1,\ldots,\lambda^{(i)}_n)\in
\RR^n$ for $i=0,\ldots,m$. Then, the following statements are
equivalent:
\begin{enumerate}
\item There exists $A_i\in \matsa$ with
$\lambda(A_i)=\lambda_i$ for $0\leq i\leq m$ and such that
$$
A_0=A_1+\ldots+A_m \ .
$$
\item For each $r\in \{1,\ldots,n\}$ and $(J_0,\ldots,J_m)\in
LR_r^{\,n}(m)$ we have \begin{equation}\label{comp ec} \sum_{j\in
J_0}\lambda^{(0)}_j\leq\sum_{i=1}^m\sum_{j\in J_i}\lambda^{(i)}_j
\end{equation} plus the condition $\suml_{j=1}^n\lambda^{(0)}_j=
\suml_{i=1}^m\suml_{j=1}^n\lambda^{(i)}_j$.
\end{enumerate}
Moreover, if $(A_i)_{i=0}^m$ are as in item 1. above and $(J_0,\ldots,J_m)\in
LR_r^{\,n}(m)$ satisfy equality in \eqref{comp ec}, then there exists a subspace $L\subseteq \CC^n$ with dim $L=r$, that simultaneously reduces $A_i$ for $0\leq i\leq m$ and such that $\lambda(P_L\,A_i)=(\lambda^{(i)}_j)_{j\in J_i}$, where $P_L$ denotes the orthogonal projection of $\CC^n$ onto $L$.
\QED
\end{teo}

\pausa
We shall refer to the inequalities in \eqref{comp ec} as
{\it Horn-Klyachko's compatibility inequalities}.

\section{On the existence of tight fusion frames.}

\pausa
The following facts exemplify the difference between the theory of vector frames 
and that of frames of subspaces. In \cite{[BF]} (see also \cite{casazza2} and  
\cite{MR}) it is shown that the local minimizers of the frame potential on the set 
$$
F^1_{m, \, n}=\{\cF=\{f_i\}_{i \in \IM}\ 
: \mbox{ each $f_i \in \cene$ \ \  
and } \  \tr(S_\cF)=\sum_{i \in \IM}\,  \|f_i\|^2=1\,\} 
$$ 
are tight frames. Since the set $F^1_{m, \, n}$ is compact and the frame potential 
is a continuous function, there must be global (and hence local) minima of the frame potential. 
This was used to give an indirect proof of the existence of such frames in the vectorial case.

\subsection{Dimensional restrictions}
Let $\mathbf d\in \N^m$ with $\tr \d\ge n$ and consider the set 
$\cS_{m ,\,  n}^1(\d) $ defined in Eq. \eqref{con tr1}. 
Using Remark \ref{de FS a vec} it follows that if $\d = \uno_m\,$, 
then we can identify $\cS_{m ,\,  n}^1(\d) $ with $F^1_{m, \, n}$, and the previous comments can be applied. 
Hence it seems natural to ask whether there always exist TFF's  
in $\cS_{m ,\,  n}^1(\d) $, since they would be all the global minimizers of 
the FF-potential on $\cS_{m ,\,  n}^1(\d) $. 
The following results show that in general the answer  is no.

\begin{pro}\label{restric}
Let $(\mathbf d,w)\in \N^m\times\RR_{>0}^m$ be a normalized pair, with $M = \tr \d \ge n$, 
 and assume that $\cW_w \in \cS_{m ,\,  n}^1(\d)$ is a TFF, 
 so that $S_{\FS} = \frac{1}{\,n} \ I_n \,$. If there exists $i\in \IM$ such that 
$$
M-d_i = \suml_{k\neq i}d_k\leq n-1 \ \implies  \ 
w_i^2 = \frac{1}{\,n} \peso{and}  P_{W_i} \, P_{W_j} =0  \peso{for every  \ \ $j\in \IM \setminus \{i\}$ .}
$$
\end{pro}

\begin{proof} 
Consider the tight vector frame $\cF=\{w_i^2\,e_j^{(i)}:\ i \in \IM \, ,\  j\in\IN{ d_i}\}$ 
associated to $\FS\,$, as described  in Remark \ref{de FS a vec}. 
Let $G \in \matrec{M}^+$ denote 
the Gramm matrix of the vector frame $\cF$ and let $G_i = w_i^2 \, I_{d_i}$ 
 denote the Gramm matrix of each subsequence  $\{w_i^2\,e_j^{(i)}:\  j\in\IN{ d_i}\}$.
 Then each  $G_i$ is a $d_i\times d_i$ principal sub-matrix of $G$.
By Cauchy's interlacing principle \cite{Ba} we  get: 
$$  
\la_j(G)\geq \la_j(G_i)\geq \la_{M-d_i+j}(G) \ \peso{for}  1\leq j\leq d_i \ ,
$$ 
where $\lambda(G)=(\lambda_j(G)\,)_{j\in \IN{M}}$ (resp $\lambda(G_i)\in \RR^{d_i}$) denotes the 
vector of eigenvalues of $G$ (resp. $G_i$) 
counting multiplicities and with its entries arranged in non-increasing order. 
By assumption, 
$$
\peso{$\lambda_j(G)=\frac{1}{\,n}$  \ \ for $j \in \In$  \ ,   \ \ $\lambda_j(G)=0$ 
 \ \ for  \ \ $n<j\leq M$ , \ and  \ $\lambda_j(G_i)=w_i^2$  \ for \ $j\in \IN{d_i}$ .}
$$ 
Thus if $\suml_{k\neq i}d_k   \leq n-1$,  then $M-d_i+1 \le n$ and 
$\frac {1}{\,n} = \la_{M-d_i+1} (G)  \le \lambda_1(G_i) = w_i^2 \le \lambda_1(G) = \frac{1}{\,n}\ $. 
It is known that, in this case,  each of the vectors $e_j^{(i)}$, $1\leq j\leq d_i\, $, 
must be orthogonal to every other vector in the $\frac{1}{\,n}$ -tight vector frame $\cF$, which 
implies the last assertion of the theorem.
\end{proof}

\begin{exa}[About the existence of tight frames in 
$\cS_{m ,\,  n}^1(\d)$]\label{existencia de tights} Consider now $\mathbf d=(2,2)$ and assume that there exists 
$\cW_w\in \cS_{2,3}^1(\mathbf d)$ that is tight. That is, we assume that there exist 
two subspaces $W_i\subset \CC^3$ with $\dim W_i=2$, $i=1,2$ and 
$w_1\, ,\,w_2 \in \R_{>0}$ such that $\frac{1}{3}\, I_3=w_1^2\,P_{W_1}+w_2^2\,P_{W_2}\,$. Since 
$d_1\, , \, d_2\leq 3-1$ we conclude from  \prop{restric} that $w_1^2 = w_2^2 =\frac{1}{3}\,$,  
and $P_{W_1}\,P_{W_2}=0$, 
which is impossible. 
This argument can be extended to show that if the 
choices of $\mathbf d=(d_i)_{i \in \IM}\, $ are such that each $d_i$ is relatively 
small compared with $n$ and $\suml_{k\neq i}d_k$ then there are no TFF's  
in $\cS_{m, \, n}^1(\mathbf d)$. For example, in  
$\cS_{k,\, 2k-1}^1(2 \cdot \uno _k), \  \cS_{3,\, 7}^1(3,3,3) , \   \cS_{3,\, 9}^1(4,4,4)$, etc, 
there are no TFF's.    \EOE 
\end{exa}

\begin{rem}\rm 
The previous results show some dimensional restrictions for the 
existence of TFF's in $\cS_{m ,\,  n}^1(\d)$. 
In the paper by Casazza and Fickus \cite{CF}, some sufficient conditions on $n$, $m$  and $\d$ 
are given (particularly if $\d$ is a multiple of $\uno_m$), 
which assure the existence in $\cS_{m ,\,  n}^1(\d)$  of such fusion frames. 
For further results in this direction, see also \cite{GK}. \EOE
\end{rem}

\subsection{Characterizations for fixed weights}
The following theorem gives us some general bounds for $\mathbf P_q(\cdot)$ and states several conditions on  $\FS \in \cB_{m, \, n}(\mathbf d)$ which 
are equivalent to the assertion that $\FS$ is a $\frac{1}{\,n}\,$-TFF. 
\begin{teo}\label{propopot}
Let $\cW_w \in \cB_{m, \, n}(\mathbf d)$, 
where $\suml_{i \in \IM}\, w_i^2\,d_i \geq 1$. 
Then 
\begin{equation}\label{hora de comer}
\frac{1}{n^2}\,I\prec_w \mathbf P_q(\cW_w) \ .
\end{equation}
For every u.i.n. $\|\cdot\|$ on $\matrec{n}$ with associated symmetric gauge
function $\psi$ we have that 
\begin{equation}\label{desiuin} \frac{1}{n^2}\  \psi( \uno)\leq \|\mathbf
P_q(\cW_w)\| \ .
\end{equation}
For every increasing convex function $f:\RR_{\geq 0}\rightarrow
\RR$ with $f(0)=0$ we have \begin{equation}\label{ultimomomento}
n\cdot f\big( \, \frac{1}{n^2} \, \big)\leq \tr\, f(\mathbf P_q(\cW_w)\, ) \ .
\end{equation}
Finally, the following conditions are equivalent: 
\ben
\item $\cW_w$ is  a $\frac{1}{\,n}\,$-TFF.
\item Majorization holds in \eqref{hora de comer}.  
\item There exists u.i.n. $\|\cdot\|$ such that equality holds in \eqref{desiuin} 
\item There exists an increasing  strictly convex function $f:\RR_{\geq 0}
\rightarrow \RR_{\geq 0}$ with $f(0)=0$ such that equality holds in
\eqref{ultimomomento}.
\een 
\end{teo}

\begin{proof}

Since $\tr(S_{\cW_w})=\suml_{i \in \IM}\,  w_i^2 \,d_i\geq 1$ then (see Example \ref{ejem}) 
it follows that $\frac{1}{\,n}\,I_n\prec_w S_{\cW_w}$ and hence 
\begin{equation}\label{tempranillo} \frac{1}{\,n}=\tr \, (\frac{1}{\,n}\,I_n)^2 \leq \tr S_{\cW_w}^2 =
\tr \mathbf P_q(\cW_w) \ \ 
\implies \ \ \frac{1}{n^2}I_n\prec_w \mathbf P_q(\FS) \ .
\end{equation}Notice that by Theorem \ref{props submayo y Klyachko} then \eqref{desiuin} and \eqref{ultimomomento}
 are consequences of this last fact.
Assume that majorization holds in \eqref{hora de comer}, so then we
have 
$$\tr \, \frac{1}{n^2} \, I_n =\tr \, \mathbf P_q(\cW_w)=\tr S_{\cW_w}^2  \ .
$$ 
Since $\frac{1}{\,n}I_n\prec_w S_{\cW_w}$ and
the function $f(x)=x^2$ is strictly convex, by Theorem \ref{props submayo y Klyachko} we conclude that there exists a unitary $U\in \cU(n)$ such
that $S_{\cW_w}=U^*(\frac{1}{\,n}\,I_n)U=\frac{1}{\,n}\,I_n\, $. On the
other hand, if there exists an u.i.n. $\|\cdot\|$ such that equality
holds in \eqref{desiuin} then, using the right-hand side of
\eqref{tempranillo} we get
 \begin{equation}
 \frac{1}{n^2}\ \psi(\uno)=\|\mathbf P_q(\cW_w)\|\geq \frac{\tr \, \mathbf P_q(\cW_w) }{n} 
 \ \psi(\uno)\geq   \frac{1}{n^2}\ \psi(\uno) \ , 
 \end{equation}
 which implies that $\tr \, (\frac{1}{\,n}\,I_n)^2 =\tr \, \mathbf  P_q(\cW_w)$. 
 As before, we conclude that $S_{\cW_w}=\frac{1}{\,n}\,I_n \, $. Similarly, if there exists an increasing  strictly convex function $f:\RR_{\geq
0}\rightarrow \RR_{\geq 0}$ with $f(0)=0$ such that equality holds in
\eqref{ultimomomento} then Theorem \ref{props submayo y Klyachko} and the right-hand side of
\eqref{tempranillo} imply that $S_{\cW_w}=\frac{1}{\,n}\,I_n\,$.  Finally, it is clear that in case $\cW_w$ is a TFF
then $\mathbf P_q(\cW_w)=\frac{1}{\,n}\ I_n \, $. The last part of the theorem follows from this fact.
 \end{proof}

\begin{teo}\label{con suf y nec}
Let $(\mathbf d,w)\in \N^m\times\RR_{>0}^m$ be a normalized pair.
Then, the following statements are equivalent:
\begin{enumerate}
\item\label{item 1} There exists $\cW_w=
(w_i \, ,W_i)_{i \in \IM} \in \cS_{m, \, n}^1(\mathbf d)$ which is a $\frac{1}{\, n}\,$-TFF. 
 \item\label{item 3} For every $1\leq r\leq n-1$ and every
$(J_0,\ldots,J_m)\in LR_r^{\,n}(m)$ we have that
\begin{equation*}\label{la ec posta}
\frac{r}{n}\leq \suml_{i \in \, \IM}\,  w_i^2\cdot |\,J_i\cap \{1,\ldots, d_i\} \,|.
\end{equation*}
\item\label{item 4} There exists an orthogonal projection $P\in \cM_m(\matrec{n} \, )$
 with $\tr(P)=n$ and such that, if $\frac{1}{\,n}\,P=(w_i\,w_j\,P_{ij})_{i,j\in \,\IM}$ with 
 $P_{ij}\in \matrec{n}$  for $i,j\in \IM\,$, then
 $$
 \frac{1}{w_i^2\,n}\ P_{ii}=\left( \, \frac{1}{w_i^2\,n}\, P_{ii}\, \right)^* =
 \left( \, \frac{1}{w_i^2\,n}\, P_{ii}\, \right) ^2 
 \ \text{ and } \tr\left( \, \frac{1}{w_i^2\,n}\,P_{ii} \, \right) = 
 d_i \ \text{ for every } \ i \in \IM \ \ .
 $$
\end{enumerate}
\end{teo}

\begin{proof}
Notice that the condition in \eqref{item 1} is equivalent to the existence of orthogonal 
projections $\{P_i\}_{i \in \IM}\, $ such that $\tr(P_i)=d_i$ for $1\leq i\leq m$ and 
such that $\suml_{i \in \IM}\,  w_i^2\,P_i=\frac{1}{\,n}\,I$. Hence, by Theorem \ref{teoK} 
it follows that condition 
\eqref{item 3} should hold, since these are Horn-Klyachko's compatibility 
inequalities for the spectra of $\{w_i^2\,P_i\}_{i \in \IM}\, $ and $\frac{1}{\,n}\,I_n\, $. 
The converse of the previous implication also follows from Theorem \ref{teoK} since a 
self adjoint $A$ operator with  $\lambda(A)=(\alpha,\ldots,\alpha,0,\ldots,0)\in \RR^n$ 
is necessarily of the form $A = \alpha\,P$ for some projection $P\in \matrec{n}$.

Assume now \eqref{item 1} and let $P_i=P_{W_i}$ for $1\leq i\leq m$, where $\cW_w=(w_i,W_i)_{i \in \IM}\, $ 
is a $\frac{1}{\,n}$-TFF. Let us consider $V_i\in \mat$ a partial isometry such that $V_i^*V_i=P_i$, 
$1\leq i\leq m$. Define $V^*=[w_1V_1^*|w_2V_2^*|\cdots|w_mV_m^*]\in \matrec{n,m\cdot n}$ and notice that $V^*V=\suml_{i \in \IM}\, w_i^2V_i^*V_i=\frac{1}{\,n}\, I_n$. Hence $VV^*=(w_iw_jV_iV_j^*)_{i,\,j=1}^m\in \cM_{m}(\matrec{n}\, )$ is such that $VV^*=\frac{1}{\,n}\,P$ for an orthogonal projection $P=(P_{ij})_{i,\,j}^m\in \cM_{m}(\matrec{n}\, )$. By comparing the diagonal blocks we get that
$$ 
\frac{1}{\,n}\, P_{ii}=w_i^2\,V_iV_i^* \ \ \implies \  \frac{1}{w_i^2\,n}\,P_{ii}=\,V_iV_i^* \ , \peso{for every}
i \in \IM \ \  .
$$
Conversely, if $P=(P_{ij})\in \cM_{m}(\matrec{n}\, )$ is as in \eqref{item 4} then there exist matrices 
$V_i\in \matrec{n}$ such that if $V^*=[w_1V_1^*|\cdots|w_mV_m]\in \matrec{n,m\cdot n}$ 
then $P=VV^*$ (since rank $ P=n$). But then, comparing the block diagonal entries we get that 
$\frac{1}{\,n} \, V_i^*V_i$ is an orthogonal projection with $\tr(\frac{1}{\,n} \, V_i^*V_i)=d_i$
for $1\leq i\leq m$. Hence $V^*V=I_n \,$, that is 
$$
\suml_{i \in \IM}\,  w_i^2\, \frac{1}{\,n} \, V_i^*V_i =\frac{1}{\,n}\,I_n \ \ .
$$ 
Therefore, if we let $W_i$ be the range of $V_i^*$ we get that $P_{W_i}=\frac{1}{\,n} \, V_i^*V_i$ so then $\cW=(w_i,W_i)_{i \in \IM}\, $ is a $\frac{1}{\,n}$ -TFF for $\CC^n$. 
\end{proof}

\section{Minimization for fixed weights}

\subsection{Lower bound for the potential}
In this subsection we translate, using Remark \ref{de FS a vec}, some well known 
results about vector 
frames (see \cite{casazza2} or \cite{RS}) to the FF context. An interesting fact is that 
there is  a notion of irregularity, defined in terms of the parameters of a given FF, which 
agree with the vectorial $n$-irregularity of their associated vector frames. 
Nevertheless, the lower bound obtained for the FF-potential is not always 
attained in the set $\cB_{m, \, n}(\mathbf d,w)$ (see Example \ref{noattine}) . 

\begin{fed} \rm Given a pair 
$\mathbf d\in  \N^m$ and $w = w^\downarrow\in \RR_{>0}^m\,$, 
 consider its {\it q-irregularity} defined as
$$
J_0(\d , w)=\max \ \left\{\, j \in \IM \ :\ 
\big( n-\suml_{i=1}^j d_i \big)\,w_j^2 > \suml_{i=j+1}^m w_i^2\,d_i \right\} \ ,
$$ 
if this set is not empty, otherwise $J_0(\d , w)=0$.
\end{fed}

\begin{pro}\label{la irregularidad fusion} 
Let $(\mathbf d,w)\in \N^m\times\RR_{>0}^m$ be a normalized pair, where $w=w^\downarrow$.
Recall that $M = \suml_{i \in \IM}\, d_i\geq n$. Let 
$
j_0:=J_0(\d , w) $ and $c=
\frac{\suml_{i=j_0+1}^m w_i^2\,d_i}{n-\suml_{i=1}^{j_0} d_i} \ < w_{j_0}^2 \ .
$ 
If $\cW_w \in \cB_{m, \, n}(\mathbf d,w)$  then 
\begin{equation}\label{en el IAM}
\FP(\cW_w)\geq \sum_{i=1}^{j_0} d_i\, w_i^4+(n-\sum_{i=j_0+1}^{m}d_i) \ c^2.
\end{equation}
Moreover, equality holds in \eqref{en el IAM} if and only if the following two conditions hold:
\ben 
\item $P_{W_i}\,P_{W_j}=0$ for 
$1\leq i\neq j\leq j_0$ and 
\item $\{w_i \, , W_i\}_{i=j_0+1}^m$ is a TFF for 
\rm $ \text{span}\{W_i:\ 1\leq i\leq j_0\}^\perp$. \it 
\een
\end{pro}

\begin{proof}
1. 
Let $\cW_w \in \cB_{m, \, n}(\mathbf d,w)$ and let  $\cF=\{w_i^2\,e_j^{(i)}:
\ i \in \IM \, ,\  j\in\IN{ d_i}\}$  be  
an associated vector frame, as described  in Remark \ref{de FS a vec}. 
Let $\mathbf a\in \R^{M}$ denote the vector whose coordinates are the norms of the 
elements of $\cF$ arranged in non-increasing order. Then, 
$$
w_k = a_{M(k,j)} \ , \peso{where \ \ $M(k, j) = \sum_{i=1}^{k-1} d_i+j$ 
\, ,\ \   for }\ k \in \IM\ \text{ and } \ j \in \IN{d_k} \ .
$$
We now consider the $n$-irregularity $r_n(\mathbf a)$ of the vector $\mathbf a$ :
$$
r_n(\mathbf a) = \max  \Big\{ j\, \in  \, \IN{n-1} \ : \ 
(n-j)a_j>\suml_{i=j+1}^M a_i\, \Big\} \ , $$
if the set on the right is non empty, and $r_n(\mathbf a)=0$ otherwise. It is  straightforward that $r_n(\mathbf a)=\suml_{i=1}^{j_0}d_i$ in the first case. Therefore, 
inequality \eqref {en el IAM} can be deduced from \cite[Theorem 10]{casazza2} 
(see also \cite{RS}). 
The same result of \cite {casazza2} 
shows that equality in Eq. \eqref {en el IAM} implies that 
$S_1=\{e^{(i)}_j: 1\leq i\leq j_0\, ,\  j\in\IN{d_i}\}$ is an orthonormal 
system in $\C^n$, and $S_2 = \{e^{(i)}_j: j_0+1\leq i\leq m \, ,\ j\in\IN{ d_i}\}$
is a tight frame for $\cS^\perp$, where $\cS = \gen{S_1} = \gen{W_i:\ 1\leq i\leq j_0}$. 
\end{proof}

\pausa
The following example shows that the lower bound in \eqref{en el IAM}  is not sharp in general.  

\begin{exa}\label{noattine}
If we set $n=3$, $\mathbf d=(2,2)$ and $w_1=w_2=\frac{1}{2}$ then $\suml_{i \in \IM}\,  w_i^2\,d_i=1$ and $J_0(\d , w)=0$. Therefore, by Theorem \ref{la irregularidad fusion} the equality \eqref{en el IAM} holds only in tight fusion frames. Still, there are no tight FS in $S_{2,3}(\mathbf d,w)\subset S_{2,3}(\mathbf d)$ since the  Example  \ref{existencia de tights}  shows that there are no tight FS in the (bigger set) $S_{2,3}(\mathbf d)$. 
\EOE 
\end{exa}

\subsection{Structure of local minima: The geometrical approach}

In what follows we consider a perturbation result for Bessel sequences of subspaces. We begin by considering some well known facts from differential geometry that we shall need below. In what follows we consider the unitary group $\cU(n)$ together with its natural differential geometric (Lie) structure. 
It is well known that the tangent space $\mathcal{ T}_{I_n}\,\cU(n)$ at the identity 
can be naturally identified with the real vector space 
$$
\mat_{ah} = i \cdot \mat _{sa} = \big\{ X \in \mat : X^* = -X \big\} \ ,  
$$ 
of anti-hermitian matrices. 
Given $G\in  \mat^+$ we consider the smooth map 
\beq\label{psiG}
\Psi_G:\cU(n)\rightarrow \cU(G)\inc \mat   \peso{given by} \Psi_G(U)=U^*GU \ , \quad U \in \matu \ ,
\eeq
where $\cU(G)$ is the unitary orbit of $G$. Under the previous identification, 
the differential of $\Psi_G$ at a the point $I_n\in \cU(n)$ in the direction given by 
$X\in \mat_{ah}$ is given by 
\begin{equation}\label{prediffi}
(D\Psi_G)_{I_n} (X)= XG -GX = [X,G] \ .
\end{equation} 
It is well known that the map $\Psi_G$ is a submersion of $\matu $ onto $\cU(G)$. Therefore, 
the differential $(D\Psi_G)_{I_n}$ is an epimorphism, 
and hence \eqref{prediffi} gives us a description of the tangent space of the manifold 
$\cU(G)$ at the point $G$: We have that $T_G\cU(G) = \big\{ \, [X, G] : X \in \mat_{ah} \, \big\}$.  

\pausa
Let us fix some notations. We denote by 
$$
\matsauno = \{A\in \matsa : \tr A = 1\} \peso{and} 
\matsao = \{A\in \matsa : \tr A = 0 \} \ .
$$
Observe that $\matsauno$ is an affine manifold contained in the real vector space 
$\matsa\,$, whose tangent space is the subspace $\matsao\,$. 
On the other hand, given $X, Y \in \matsa\,$, it is easy to see that 
$\tr XY \in \R$. Therefore the inner product $\api A,\ B \cpi = \tr \, B^*A $ 
of $\mat$ still works as a real inner product on $\matsa\,$. 

\pausa
Given a set $\{P_j :j\in \IM\}\inc \matsa$ of projections, we denote by 
\beq\label{conm}
\{P_j :j\in \IM\}' = \{A\in \mat : AP_j = P_j \, A \peso{for every} j \in \IM\}\ .
\eeq
Note that $\{P_j :j\in \IM\}'$ is a closed selfadjoint subalgebra of $\mat$.
Therefore, the algebra 
$\{P_j :j\in \IM\}' \neq \C\, I_n \iff $ there exists 
a non-trivial orthogonal projection $Q \in \{P_j :j\in \IM\}'$. 

\medskip
\begin{teo}\label{sec loc} Let $\msfram \in \cB_{m, \, n}^1(\d, w)$  be a BSS. Denote 
$P_j=P_{W_j}$ for every $j\in \IM\,$.  
Let $\Psi:\matu^m\rightarrow \matsauno \inc \matsa $ be the smooth function given by 
$$ 
\Psi(U_1,\ldots,U_m)=\sum_{j=1}^m w_j^2 \, U_j^* P_j\,U_j 
=\sum_{j=1}^m w_j^2 \, \Psi_{P_j} (U_j)  \ \  , \peso{for}
 (U_1,\ldots,U_m) \in \matu^m \ .
$$ 
Then the following conditions are equivalent:
\ben
\item The differential of $\Psi$ at $I:= (I_n,\ldots,I_n)\in \matu^m$ 
is {\bf surjective} 
\item $\{P_j :j\in \IM\}' =  \C\, I_n \,$. 
\een 
In this case, the image of $\Psi$ contains an open neighborhood of $\Psi(I)
= \suml_{j=1}^m w_j^2 \,P_j\,$ in $\matsauno\, $, and 
$\Psi$ admits smooth (and hence continuous) local cross sections around $\Psi(I)$. 

\end{teo}
\begin{proof}
It is clear from its definition that $\Psi$ is a smooth function. Moreover, under the previous 
identification $T_{I_n} \cU(n) = \mat_{ah}\,$, and using Eq. \eqref{prediffi}, we can see that \begin{equation}\label{fora}
D\Psi_{I}(X_1,\ldots,X_m)=
\sum_{j=1}^m w_j^2\, [X_j,  P_j ] \  \ ,\peso{for}  (X_1,\ldots,X_m) \in \mat_{ah}^m \ .
\end{equation}
The tangent space of $\matsauno$ is the real vector space $\cM_n^0(\CC)_{sa}\,$, which  has a natural inner product given by $\langle Y,Z\rangle=\tr(YZ)$.
Denote by  $T = D\Psi_I $ and assume that $T$ 
is not surjective. Then there exists $0\neq Y\in \cM_n^0(\CC)_{sa} $ which  is orthogonal to the 
image of $T$. Using Eq. \eqref{fora} we deduce that, for every  $(X_1,\ldots,X_m) \in \mat_{ah}^m\,$,  
it holds that 
\begin{equation}\label{ortogo}
0=\big\langle T(X_1,\ldots,X_m), Y\big\rangle = 
\sum_{j=1}^mw_j^2\,\tr \big(\, [X_j\, ,P_j]\,Y \,  \big) =
\sum_{j=1}^mw_j^2\,\tr\big(\, X_j\,[P_j\, ,Y]\,  \big)\ . 
\end{equation}
Since each $[P_j\,,Y]\in \mat_{ah}\,$,  we can choose each  $X_j=[P_j\,,Y]$,  
and so Eq. \eqref{ortogo} implies that $[P_j\, ,Y]=0$ for 
every $ j\in \IM\,$. 
In other words,  that $Y\in \{P_j :j\in \IM\}'$. 
On the other hand, since $0\neq Y \in \matsa$ and $\tr Y = 0$,  
then $Y \notin \C \, I_n\,$. 
The converse follows from the previous argument, by taking 
$Y \in \{P_j :j\in \IM\}'$ such that $0 \neq Y = Y^*$ and $\tr Y = 0$.
\end{proof}

\pausa
Let $(\mathbf d,w)\in \N^m\times\RR_{>0}^m$ be a normalized pair.
We shall consider on $\cB_{m, \, n} (\d , w)$ the distance 
$$
d_P(\cW_w\, ,\cW'_w)=\max_{i\in \IM} \|P_{W_i}-P_{W_i' }\|  \ 
$$ 
(recall that the weights are fixed), called $punctual$, and the pseudo-distance 
$$ 
d_S(\cW_w\, ,\cW'_w)=  \|S_{\cW_w}-S _{\cW_w' }\|  \ ,
$$ 
called $operatorial$. The problem of finding $local$ minimizers for the FF-potential can 
be stated for anyone of those distances between FF's. 

\pausa

\begin{cor} \label{sec loc 2}
Let $(\mathbf d,w)\in \N^m\times\RR_{>0}^m$ be a normalized pair.
Assume that $\FS\in \cB_{m, \, n} (\d , w)$ satisfies that $\{P_{W_j} :j\in \IM\}' = \C \, I_n\,$.
Then $\cW_w$ is a FF and the map 
$$
S : \cB_{m, \, n} (\d, w) \to \matsauno  \peso{given by} S(\cV_w) = S_{\cV_w} 
= \suml_{i\in \IM} w_i^2 P_{V_i} \ ,
$$
for $\cV_w = (w_i,V_i)\subim \in \cB_{m, \, n} (\d, w) $, satisfies that 
\ben
\item The image of $S$ contains an open neighborhood of $S_{\FS}$ in $\matsauno\, $.
\item $S$ has 
$d_P$-continuous local cross sections around $S_{\FS}\,$.
\een
\end{cor}
\proof Notice that the condition  $\{P_{W_j} :j\in \IM\}' = \C \, I_n\,$ implies that $\cW$ 
is a generating sequence of subspaces. To prove the properties of the map $S$, 
just compose a local cross section for the map $\Psi$ of Theorem \ref{sec loc} 
(which is open in $I$) with 
the map $\Phi: \matu^m \to \cB_{m, \, n} (\d)$ given by
$$ 
\Phi (U_1\, , \dots , U_m ) = (w_i \, , U_i (W_i)\,)_{i\in \IM}  = 
(w_i \, , R(U_i P_{W_i} U_i^*)\,)_{i\in \IM} \ . 
$$
Observe that $S \circ \Phi = \Psi$, so that $S$ is open in $\Phi(I) = \FS$. \QED

\begin{rem}\label{desc alter}
Here is an alternative statement of Corollary \ref{sec loc 2}: Under the same assumptions and 
notations about $\FS$, it holds that  $S_{\cW_w}\in \matinv^+$ and, for every sequence 
$(S_k)_{k\in \N}$ in $ \matsauno$ such that 
$S_k \convk S_{\cW_w}\,$,  there exists a sequence  $(\cV_k)_{k\ge k_0}$ 
in $\cS_{m, \, n} (\d\, , \, w)$ such that 
$$ 
d_P(\cV _k \, , \, \FS ) \convk 0 \peso{and}  S_{\cV_k} = S_k  \peso{for every} k_0 \le k \in \N \ . 
$$
This formulation of Corollary \ref{sec loc 2} generalizes \cite[Thm 5.3]{MR} to the context of fusion frames with fixed weights. \EOE
\end{rem}

\pausa
It is not clear that a $d_p\,$-local minimizer for the FF-potential on $\cB_{m, \, n} (\d , w)$ 
must be a fusion frame, i.e. its frame operator is an invertible operator. The following Lemma 
shows that this is true.

\begin{lem}\label{los minimos son frames}
Let $(\mathbf d,w)\in \N^m\times\RR_{>0}^m$ be a normalized pair. 
Let $\cW_w$ be a $d_p$-local minimizer for the FF-potential in $\cB_{m, \, n}(\mathbf d,w)$. 
Then $S_{\cW_w}$ is invertible $($equivalently, $\cW_w\in \cS_{m, \, n}(\mathbf d,w)$$\,)$.
\end{lem}

\proof
Suppose that $S_{\FS}$ has nontrivial nullspace $N(S_{\FS})$.
If $x \in N(S_{\FS})$
\begin{align}\label{pepe}
0 &=\pint{S_{\FS}\, x}{x}=\sum_{j\in\IN{m}}\,\pint{w_j^2P_{W_j}\, x}{x}
=\sum_{j\in\IN{m}}\,w_j^2\,\|P_{W_j}\, x\|^2  \ .
\end{align} 
In other words, $W_i \inc R(S_{\FS})$ for every $i \in \IM\,$. Since 
$\tr(\mathbf d)\ge n> \dim R(S_{\FS})$, we deduce that there exists $i \ne j $ in $ \IM$ such that 
$P_{W_j} \, P_{W_i} \ne 0$. Fix that pair $i, \, j$. Fix also  $f \in W_i \setminus W_j^\perp $ and 
 $g \in N(S_{\FS})$  two unit vectors. For every $t\in [0, \pi/2]$, 
take the unit vector  $g(t) = \cos t \cdot f + \sin t \cdot g$. 

\pausa
Let $V_i = W_i \ominus  \ \gen{f}$,  $W_i (t) = V_i \oplus \gen {g(t)}$ and 
$\FS(t) \in \cB_{m, \, n}(\mathbf d,w)$ the sequence obtained 
by replacing $W_i$ by $W_i(t)$ in $\FS\,$.  
As $g \in W_k^\perp$ for every $k \in \IM\,$,  for every $t\in (0, \pi/2]$,  
$$
\barr{rl}
\frac{\FP(\FS)- \FP(\FS(t)\,)}{2}  & =   
\suml_{k=1}^m w_i^2w_k^2 \Big( \tr(P_{W_i}\, P_{W_k}) -  \tr(P_{W_i(t)}P_{W_k}) \, \Big) \\
& =\suml_{k=1}^m w_i^2w_k^2 \Big( \tr(P_{W_i}\, P_{W_k}) -  
\tr \big( g(t) g(t)^* \,  P_{W_k}\big) - \tr(P_{V_i}P_{W_k})\, \Big) \\
& =\suml_{k=1}^m w_i^2w_k^2 \Big( \tr(P_{W_i}\, P_{W_k}) -  
\cos^2 t \, \tr \big( ff^* \,  P_{W_k} \big) - \tr(P_{V_i}P_{W_k})\, \Big) \\
	& > \suml_{k=1}^m w_i^2w_k^2 \Big( \tr(P_{W_i}\, P_{W_k}) - 
	\tr \big( ff^* \,  P_{W_k} \big) - \tr(P_{V_i}P_{W_k})\, \Big) = 0  \ ,
\earr 
$$
because $\tr \big( ff^* \,  P_{W_j} \big) = \|P_{W_j} \, f\|^2 \ne 0$. 
Hence
$\FP (\cV_w(t)\,) < \FP(\FS)$ for every $t\in (0, \pi/2]$.  
Taking  $t \to 0$, we have that $\FS(t) \overset{d_P}{\to} \FS$, and this contradicts the minimality
of $\FS$. \QED

\pausa
Given $S\in \matsa$ with $\sigma (S) = \{\mu_1\, , \dots, \,\mu _r\}$, 
we denote by $P_{\mu_k}(S) = P_{N(S-\mu_k \, I_n)}\in \matpos$, the spectral projection of $S$ 
relative to $\mu_k\,$, for $k \in \IN{r}\,$.  These projections satisfy that
\ben
\item $P_{\mu_k}(S)\, P_{\mu_j}(S) = 0$ if  $k\neq j$, 
and $\suml_{k=1}^p P_{\mu_k}(S) =I_n\,$
(i.e., they are a system of projectors). 
\item For every $k \in \IN{r}\,$, it holds that $S \,  P_{\mu_k}(S) = \mu _k \, P_{\mu_k}(S)$, 
so that $S = \suml_{k=1}^p \mu _k \, P_{\mu_k}(S) $. 
\een
The following theorem generalizes a similar result given in the paper by Casazza and Fickus 
\cite[Theorem 4]{CF}, for the case 
$w = \uno_m\,$. Nevertheless, our approach is based on completely different techniques.

\begin{teo}\label{conmutan}
Let $(\mathbf d,w)\in \N^m\times\RR_{>0}^m$ be a normalized pair.
Let $\cW_w  \in \cB_{m, \, n}(\mathbf d,w)$ be a local minimizer of 
the FF-potential with respect to the distance $d_P\,$. If 
$\sigma(S_{\FS} ) = \{\mu_1\, , \dots, \,\mu _r\}$, then 
\beq\label{conmu}
P_{\mu_k}(S_{\FS} )\ \in \ C_{\FS} = \{P_{W_j} :j\in \IM\}' 
\peso {for every}  k\in \IN{r}\  . 
\eeq
The same property holds 
whenever $\FS$ is a $d_P$-local minimizer 
in $\mathcal B^1_{m, \, n}(\mathbf d)$.
\end{teo}
\proof 
Recall from \lema {los minimos son frames} that $\FS \in \mathcal S_{m, \, n}(\mathbf d , w)$, 
in other words that $0 \notin \sigma(S_{\FS})$. 
Consider the set $\mathcal Q$ of finite systems of projectors  
$\{Q_k\}_{k \in \IN{p}}$ such that each $Q_k \in C_{\FS}\, $. 
Observe that  $\mathcal Q$ is not empty because $\{I_n\}\in \mathcal Q$. 
Then $\mathcal Q$ has a maximal element $\{Q_k\}_{k \in \IN{p}}$ with respect 
to the order induced  by refinement. 
 Fix $k \in \IN{p}\,$. 
For each $i \in \IM$ put  $\ese_i =  W_i \cap R(Q_k) \,$ and 
$\ene_i = W_i\cap  R(Q_k)^\perp $. Using that $Q_k \in C_{\FS}\, $, we get that 
each $W_i = \ese_i \oplus \ene_i\,$. 
 Set $r_i = \dim \ese_i  \,$ and $\mathbf {r} = (r_1 \,  ,\dots , r_m)$. 
 Then, the sequence  $\cW_{k,\,w}=(w_i\, , \ese_i)_{i\in \IM}$ is a 
 FF for $R(Q_k)$. 
 We claim that $\cW_{k,\, w}$ is a local minimizer of the FF-potential in 
 $$
 \cB(Q_k,\mathbf r , w):=\big\{\, \cV_w = (w_i \,,\, V_i)_{i\in \IM} \in \cB_{m, \, n}(\mathbf r,w)\ :\  
 V_i\subseteq R(Q_k)  \peso{for every  \ $i \in \IM \ \big\}$ .}
 $$ 
Indeed, given $\cV_w \in \cB(Q_k,\mathbf r, w)$,  
put $ \tilde \cV_w=(w_i\, , V_i\oplus \ene_i )_{i\in \IM} \in \cB_{m, \, n}(\mathbf d, w) $.
Observe that the map  $\cV_w  \mapsto \tilde \cV_w $ preserves the distance $d_P\,$. 
Moreover, since $Q_k \in C_{\FS}\, $, then  each $P_{\ene_i} =(I_n- Q_k) P_{W_i}\,$ so that,
by Eq \eqref{defi del fusion pot},  
$\FP(\tilde \cV_w)=\FP(\cV_w)+ \FP \big( \, (w_i\, ,\ene_i)\subim \big)$, 
and the second summand does not depend on $\cV_w\,$.  
Then, the  claim follows from the fact that 
$\widetilde{\cW} _{k,\, w} = \cW_w\,$.

\pausa
Observe that $S_{\FS}$ commutes with $Q_k\,$. 
We now show that $S_{\FS}\,Q_k=\alpha_k \,Q_k$ for some $\alpha_k \in \sigma(S_{\FS})\,$. 
Indeed, 
by the maximality of $\{Q_i\}_{i=1}^p$ in $\mathcal Q$, it follows that 
there is no non-trivial 
sub-projection $Q'$ of $Q_k$ such that $Q' \in \{P_{\ese_j}\,  : \, j\in \IM\}'$. Then we can apply 
Corollary \ref{sec loc 2}  (taking $\hil = R(Q_k)$ and renormalizing the traces) to show 
that every positive operator (with the correct trace) near $S_{\FS}\,Q_k$ 
has the form $S_{\cV_w}$ for some $\cV_w \in \cB(Q_k,\mathbf r , w)$ close to 
$\cW_k\,$. But if $S_{\FS}\, Q_k$ is not a scalar multiple of $Q_k\,$, then 
we can choose $S_{\cV_w}$ in such a way that 
$$
\FP(\cV_w )= \tr S_{\cV_w}^{\, 2}<\tr\, \big( \, S_{\cW_w}^{\, 2}  Q_k \, \big) 
= \FP(\cW_k) \ .
$$
But this contradicts the fact that $\cW_k$ is a local minimizer of the 
FF-potential in  $\cB(Q_k,\mathbf r,w)$. Hence $S_{\FS}\,Q_k=\alpha_k \,Q_k\,$
and $Q_k \le P_{\alpha_k}(S_{\FS})\,$. 
Using that  $\suml_{k=1}^p Q_k=I_n\,$, it is easy to see that 
each 
\beq
P_{\mu_i}(S_{\FS}) = \suml_{k\in \, J_i} Q_k \in C_{\FS}\ , \peso{where} 
J_i = \{ \, k\in \IN{p}  \ :  \ \alpha _k = \mu_i \, \} . \QEDP
\eeq

\begin{rem}\label{conmutan2} Next we  give two reinterpretations of Theorem \ref{conmutan}.
Under the assumptions and notations of the theorem, the following properties hold:
\ben
\item 
For each $i \in \IM\,$, there exists an o.n.b. $\cB_i = \{e^{(i)}_j : j \in \IN{d_i}\}$ of $W_i\,$, consisting of eigenvectors of $S_{\FS}\,$.  
Indeed,  observe that each $P_{W_i} = \sum _{k\in \,\IN{r}} P_{W_i} P_{\mu_k}(S_{\FS})\,$ 
and the  fact that, for a fixed $i \in \IM\,$, 
 the projections $P_{W_i} \, P_{\mu_k}(S_{\FS})\,$ are pairwise orthogonal. 
\item For each $\mu_k \in \sigma(S_{\FS})$, denote by  
$\eme_{k\, , \, i} = N(S_{\FS} -\mu_k \, I_n)\cap W_i\,$. Then, it holds that 
  the sequence 
$\cW_k = (w_i \, , \, \eme_{k\, , \, i})\subim \,$ is a {\bf tight} FF 
for $N(S_{\FS} -\mu_k \, I_n)$. This follows because  its frame operator 
$S_{\cW_k } = P_{\mu_k}(S_{\FS}) \, S_{\FS} = \mu_k \, P_{\mu_k}(S_{\FS})$. 
\EOE 
\een
 \end{rem}

\subsection{The eigenvalues of all $d_S$-minimizers coincide}
Recall that, given $S\in
\matsa\,$, we denote by $\lambda(S)\in \R^n$ the vector of the $n$ eigenvalues 
of $S$, counted with multiplicities, in such a way that $\la (S) = \la(S)^\downarrow$. 

\begin{lem}\label{unlem}
Let $(\mathbf d,w)\in \N^m\times\RR_{>0}^m$ be a normalized pair.
Then  
\ben
\item $\cB_{m, \, n}(\mathbf d,w)$ is $d_P$-compact. 
\item The set $
\Lambda_{m, \, n}(\mathbf d,w)=\{\lambda(\frs):\ \cW_w\in \cB_{m, \, n}(\mathbf d,w)\} $
is a {\bf convex} and compact subset of  \ $\R^n$. 
\item 
The set $\{\frs :\ \cW_w\in \cB_{m, \, n}(\mathbf d,w)\}$ 
is compact and closed under unitary conjugation.
\een
\end{lem}
\begin{proof}
Let $\cW_w,\, \cW'_w \in \cB_{m, \, n}(\mathbf d,w)$. 
For $t\in [0,1]$ consider $\lambda=t\,\lambda(S_{\cW_w}) +(1-t)\, \lambda(S_{\cW'_w})$ and notice that  $\lambda=\lambda^\downarrow$. Therefore, for every admissible $(m+1)$-tuple $(J_0,\ldots,J_m)\in LR_r^n(m)$, $1\leq r\leq n-1$ we have 
$$ 
\sum_{j\in J_0}\lambda_j=\sum_{i=1}^m\sum_{j\in J_i}t\,\lambda_j(S_{\cW_w})+(1-t)\,
\lambda_j(S_{\cW'_w})\, \leq \sum_{i=1}^m\sum_{j\in J_i} w_j^2\ |\{1,\ldots,d_j\}\cap J_i| \ ,
$$
since both $\lambda(S_{\cW_w})$ and $\lambda(S_{\cW'_w})$ satisfy Horn-Klyachko's inequalities. 
Hence, by Theorem \ref{teoK}, 
there exists  $\cV_w=(w_i,V_i)_{i\in \IM}\in 
\cB_{m, \, n}(\mathbf d,w)$ such that $\lambda(S_{\cV_w})=\lambda$. This shows that 
$\lambda\in \Lambda_{m, \, n}(\mathbf d,w)$,  so that $\Lambda_{m, \, n}(\mathbf d,w)$ is convex. 
The fact that the set $\{S_\cW:\ \cW_w\in \cB_{m, \, n}(\mathbf d,w)\}$ 
is closed under unitary conjugation is apparent. Finally $\cB_{m, \, n}(\mathbf d,w)$
 is $d_P$-compact because each Grassmann manifold 
 $$
 \cP_{d_i}(n) = \{P = P^*P \in \matpos : \tr P = d_i\} = \{UP_iU^* : U \in \matu\} = \cU(P_i)  \ ,
 $$ 
for every fixed $P_i \in \cP_{d_i}(n)\,$, is compact. This follows because $\matu$ is compact. 
By continuity, the other sets involved are also compact. 
\end{proof}

\begin{teo}\label{espectrae} 
Let $(\mathbf d,w)\in \N^m\times\RR_{>0}^m$ be a normalized pair.
Then, 
\ben
\item The spectra (with multiplicities) of all 
the frame operators of global minimizers of the FF-potential in 
$\cB_{m, \, n}(\mathbf d,w)$ coincide.
\item The local minimizers of the FF-potential in 
$\cB_{m, \, n}(\mathbf d,w)$ with respect to the pseudo-distance 
$d_S$ lie in $\cS_{m, \, n}(\mathbf d,w)$ and are also global minimizers.
\een
\end{teo}
\proof
 Let $\cW_w\in \cB_{m, \, n}(\mathbf d,w)$ and notice that 
$\FP(\cW_w)=\|\lambda(S_{\cW_w})\|^2$. Since $\Lambda_{m, \, n}(\mathbf d,w)$ 
is a compact convex set, then  there exists a {\bf unique} $\lambda_0\in  \Lambda_{m, \, n}(\mathbf d,w)$ 
which minimizes the euclidean norm on $\Lambda_{m, \, n}(\mathbf d,w)$. 
Hence, if $\cW_w$ is a global minimizer of the FF-potential 
in  $\cB_{m, \, n}(\mathbf d,w)$, we can conclude that 
 $\|\lambda(S_{\cW_w})\|^2\leq \|\lambda_0\|^2$, which implies that $\lambda(S_{\cW_w})=\lambda_0\,$.

\pausa
Observe that the map $\sigma : \cB_{m, \, n}(\mathbf d,w) \to \Lambda_{m, \, n}(\mathbf d,w)$
given by $\sigma (\FS) = \la (\frs) $ is continuous with respect to the pseudo-distance 
$d_S$ of $\cB_{m, \, n}(\mathbf d,w)$. Moreover, 
$\sigma $ is an open map. 
Indeed, fix $\FS\in \cB_{m, \, n}(\mathbf d,w)$,  $\la  = \la(\frs)  \in \Lambda_{m, \, n}(\mathbf d,w)$ and 
take $\mu \in \Lambda_{m, \, n}(\mathbf d,w) $ close to $\la\,$. 
Let $U \in \matu$ such that  $\frs = U \diag {\lambda \,} U^*$. By \lema{unlem}, 
there exists $\cV_w \in \cB_{m, \, n}(\mathbf d,w)$ such that $S_{\cV_w} = U \diag {\mu\,} U^*$. 
Now observe that $d_S(\FS\, , \cV_w) =  \|\frs - S_{\cV_w} \| = \|\la - \mu\|_{\infty}\,$. 

\pausa
Therefore, if $\FS $ is a $d_S$-local minimizer of the FF-potential in $\cB_{m, \, n}(\mathbf d,w)$,  then 
$\lambda(S_{\cW_w})$ is a local minimum for the euclidean norm 
in the set $\Lambda_{m, \, n}(\mathbf d,w) \inc \R^n$. 
By a standard computation, 
the convexity of $\Lambda_{m, \, n}(\mathbf d,w)$ 
implies that 
$\lambda(S_{\cW_w})$ must be  {\it the} global minimizer $\la_0\,$, 
and therefore $\FS$ is a global minimizer in $\cB_{m, \, n}(\mathbf d,w)$. \QED

\begin{conj}
Local minimizers of the frame potential in $\cB_{m, \, n}(\mathbf d,w) $ (resp. $ \cS_{m,\, n}^1(\d)\,$) with respect to   punctual distance $d_P$ are also global minimizers.
\end{conj}

\noi
In some particular cases (i.e. for particular choices of the parameters  $n$, $m$, $\d$ and $w$), 
there is an affirmative answer for this conjecture (see \cite[Theorem 5]{CF}).

\begin{rem}\label{pa no repetir}
All the previous results remain true if one replaces $\cB_{m, \, n}(\mathbf d,w) $ by $\cB_{m ,\, n}^1(\d)$  
(i.e.,  minimizing without fixing the sequence of weights). We present some of the new statements 
 without proofs,  since they
  are based on  techniques that are similar to those already developed.
\begin{enumerate}
\item As in the proof Lemma \ref{unlem}, Horn-Klyachko's compatibility inequalities 
\eqref{comp ec} show that 
$$
\Lambda_{m, \, n}^1(\mathbf d)=\{\lambda(\frs):\ \cW_w\in \cB_{m, \, n}^1(\d)\} \peso{is compact and convex.}
$$
\item This fact implies that the (ordered) spectra of the frame operators of global 
minimizers of the FF-potential in $\cB_{m, \, n}^1(\d)$ coincide.
\item Finally, the argument of the proof of \teor{espectrae} can be adapted to yield 
that local minimizers of the FF-potential in $\cB_{m, \, n}^1(\d)$,  with respect to the 
operator distance $d_S\,$, are also global. \EOE
\end{enumerate}
\end{rem}

\pausa
As in the case of fixed weights (\lema{los minimos son frames}), 
the global minimizers for the FF-potential in  $\cB_{m\,,n}^1(\d)$ are fusion frames.
\begin{pro}\label{pa no rep norma 1}
Let $\FS$ a $d_S$-local (and hence global) 
minimizer for the FF-potential in $\cB_{m, \, n}^1(\d)$. Then its frame operator
$S_{\FS} \in \matinv^+$. In other words, $\FS\in \cS_{m, \, n}^1(\d)$. 
\end{pro}

\proof
Let $J = \{i \in \IM : w_i \ne 0\}$, and $k = \suml_{ i \in J} d_i \,$. 
Note that, if $k\geq n$, we can apply \lema{los minimos son frames} (fixing the weight $w_J$) 
and we are done.

\pausa
We assume that $k < n$, and will obtain a contradiction. 
Without loss of generality, we can suppose that $J = \IN{r}\,$. 
It follows immediately that  $W_i\perp  W_j$ for 
$1\leq i\neq j\leq r$ and $w_i^2=\frac{1}{k}$  for every $1\leq i\leq r$. 
Thus, $\FP(\FS)=\frac{1}{k}\ $.
Moreover, if $d=d_{r+1}\,$, then $k+d>n$. Otherwise, if we take a subspace $W_{r+1}
\inc \big(\, \bigcup\limits\subir W_i\, \big)^\perp $ with $\dim W_{r+1}=d$ 
and we set $w_i^2=\frac{1}{k+d} \, $, for $i \in \IN{r+1}\,$, 
then we get a BSS in $\cB_{m, \, n}^1(\d)$ with FF-potential $\frac{1}{k+d}\, < \, \frac{1}{k} \ $.

\pausa
Therefore,  we can construct $\cV_v=(v_j\,, \, V_j)_{j\in \IN{r+1}}
\in \cB_{m, \, n}^1(\mathbf d)$ in the following way:
$$
v_j^{\,2} = \begin{cases} 
   a   & \mbox{if } 1\leq j\leq r \\
   b & \mbox{if }j=r+1  \end{cases}
    \peso {and} 
V_ j = 
\begin{cases} 
  \quad\quad \quad W_j   & \mbox{if } 1\leq j\leq r \\
   \big[ \,\bigoplus\limits_{i=1}^r W_i  \, \big]^\perp \,  \oplus \, T & \mbox{if }j=r+1  
  \ \, ,
\end{cases}
$$
where $T \inc \bigoplus_{i=1}^r W_i $ 
is a  subspace  with $\dim T=d+k -n$ (so that $\dim V_{r+1} = d$), 
and $k a + d b = 1$. 
It is easy to see, by taking an orthonormal basis of each subspace $V_i\,$, that
\beq 
\label{el FF del V}
\barr{rl}
f(a) &= \FP(\cV_v)= k a^2  + d b^2 + 2 (d+k-n) ab 
\peso{with} k a + d b = 1 \ \ , \ \ \ a \in \left[\, 0 \, ,  \, \frac {1}{\, k} \  \right]
\ \ . \earr
\eeq 
Easy computations show that $ f'(\frac {1}{\, k} ) = \ds\frac{2( n-k)}{d} > 0$. 
Since $f(\frac {1}{\, k}) = \frac {1}{\, k} \ $, there exist pairs $(a, b) $ such that 
the FF-potential is lower that $\frac {1}{\, k}\ $, which contradicts the minimality of $\FS$.
\QED

\pausa
Next, we summarize the facts described in Remark \ref{pa no repetir} and Proposition \ref{pa no rep norma 1}.

\begin{teo}\label{lo mesmo}\rm
Let $\mathbf d\in  \N^m$. 
Then, local minimizers of the frame potential in 
$\cB_{m, \, n}^1(\d)$ with respect to the pseudo-distance $d_S$ lie in $\cS_{m, \, n}^1(\d)$ and
are also global minimizers. Moreover, the spectra of the frame operators of these local minimizers 
coincide. \QED
\end{teo}

\section{Minimization for fixed subspaces}
In this section we shall characterize the sequences of weights which minimize 
the potential of a fixed sequence of subspaces. The main tools are some results
about Hadamard indexes of \cite{[CS]}, which we shall state in some 
detail in the Appendix. Recall that, for $A, B\in \mat$, their 
Hadamard product is the matrix $A\circ B = (A_{ij} \, B_{ij} ) _{i\, , \, j \in \In}\, \in \mat$.

 \subsection{Minimal weights}
In this section we fix $m \in \N$, a Hilbert space $\hil$ with $\dim \hil =n$, 
and $\d\in \N^m$ with $M = \tr \, \d \ge n$. We also 
fix a sequence  $\cW=\{W_i\}_{i \in \IM}\, $ of subspaces which spans $\hil$, such that each 
$\dim W_i = d_i\,$. 
Our aim is to minimize the FF-potential over all sequences $w\in \R_+ ^m$ such that 
$\FS \in \cS_{m, \, n}^1(\d)$.

\pausa Recall that the Benedetto-Fickus FF-potential of $\cW_w$ is given by:
\begin{equation}\label{defi del fusion pot2} \FP(\mathcal \cW_w)=\suml_{i,j=1}^m
w_i^2w_j^2\tr(P_{W_i}P_{W_j}) = \tr S_{\cW_w}^{\, 2} \ .
\end{equation}

\begin{fed}\rm
Let $\msfram \in \cS_{m, \, n}^1(\d)$. We denote by  
\beq 
\peso{$B= B_{\FS} \in \matmreal$  \ \  the matrix  given by 
\ \ \ $B_{ij}=w_i^2w_j^2\tr(P_{W_i}P_{W_j})$ .} \EOEP
\eeq
\end{fed}

\begin{lem}
Let $\msfram$ be a FF for  $\hil$. Then, $B_{\FS}\in \matmreal^+$. 
\end{lem}
\begin{proof}
$B_{\FS}\in \matmreal^+$ because it is the Gram matrix for the vectors 
$\{w_i^2P_{W_i}\}_{i \in \IM}\, $  in the euclidean space $\matrec{n}$ with the inner 
product defined as $\pint{X}{Y}=\tr \, Y^*X$. 
\end{proof}

\begin{Not}\label{notat}
We shall fix some notations and assumptions:
\ben 
\item
We begin with a fixed normalized sequence of weights, in the sense that 
 $$
 \peso{$\mathbf w=(w_i)_{i \in \IM}\, $  \ \ is given by} 
 w_i= d_i \mrai   \peso{for every} i\in \IN{m} \ .
 $$
Observe that 
the condition $w_i=d_i \mrai $ means that each ``vector" $w_i \, P_{W_i} $ of $\FS$ has 
size $ \|w_i \, P_{W_i}\|_{_2}=1$. 
This justifies the 
word ``normalized" of  $\FS$.
\item
Given  a sequence of weights $\mathbf a=(a_i)_{i \in \IM}\,  \in \RR_+^m$, we denote by 
$\mathbf a\cdot \FS$ the Bessel sequence of subspaces ${\bf a} \cdot \FS=(a_i\,w_i\, W_i)_{i \in \IM}\, $. 
\item 
In $\mathbf a\in \RR^m_+  \,$, then  $\tr(S_{\mathbf a \cdot \FS})=\suml_{i \in \IM}\, a_i^2 w_i^2d_i\,$.
Therefore, as we start with normalized weights, 
\beq\label{norma uno}
\pa \cdot  \FS \in \cS_{m, \, n}^1(\d) \ \iff \ \suml_{i \in \IM}\, a_i^2 w_i^2d_i = 
\suml_{i \in \IM}\, a_i^2 = 1 \ \iff \ 
 \|\mathbf a\|=1  \ .
\eeq
\item Let  $A= A_{\FS}\in \matmreal$ the matrix given by $A_{ij}= 
(B_{\FS})_{ij} \rai = w_i \, w_j \, \tr(P_{W_i}P_{W_j})\rai $, for $i,j\in \IM\,$. 
Observe that 
$A$ is selfadjoint, but possibly $A \notin \matmreal^+$. On the other hand, 
\beq \label{FPA}
\barr {rl} 
\FP( \FS) &  =\suml_{i,j=1}^m w_i^2 \, w_j^2 \, \tr(P_{W_i}P_{W_j})= \| A_{\FS} \|_{_2}^2 
\peso{and also}  \\ 
\FP(\mathbf a\cdot \FS)& =\suml_{i,j=1}^m a_i^2\, a_j^2\,
w_i^2\, w_j^2\, \tr(P_{W_i}P_{W_j})= \| \pa  \pa ^*  \circ A_{\FS}\|_{_2}^2 \ . 
\earr 
\eeq
\item  Using the previous identities, we can now define the main notion of this section:
\beq
I(\FS) : =\min\limits_{\|a\|=1} \FP(\mathbf a\cdot \FS)  = 
\min\limits_{\|a\|=1} \| \pa  \pa ^*  \circ A_{\FS} \|_{_2}^2   \ . \EOEP
\eeq
\een
\end{Not} 

 \pausa
 In order to compute $I(\FS)$ as well as to describe the set of weights $\mathbf a\in \R_{\geq 0}^m$, with $\|\mathbf \pa\|=1$ for which $I(\FS)=  \FP(\mathbf a\cdot \FS)$,  
the main tools are some results about Hadamard indexes of \cite{[CS]}, which we shall state in some 
detail in the Appendix. Here we just give  the basic definitions. 

\begin{fed}\label{IND1}\rm
Let $B\in \cM_m(\C )^+$.
The {\bf minimal}-Hadamard index of $G$ is 
\[
I(B)=\max \ \{  \ \lambda \geq 0\,  : \, B\circ C\geq \lambda C \peso{for every}  C \in \cM_m(\C )^+\} \ .
\]
For $A  \in \matm_{sa}\,$, we define the spectral and the $\| \cdot \|_{_2} \ $ Hadamard  indexes:
$$
I_{sp} (A) = \min_{\|x\|=1} \|A\circ x \, x^* \|_{sp}  \peso{and} 
I_{2} (A) = \min_{\|x\|=1} \|A\circ x \, x^* \|_{_2} \ .
$$
For a matrix $G\in \mat $, we write  $0\leqp G $ if  all entries $G_{i,j} \ge 0$. Given 
$J \inc \IN{m}$ with $|J| = k$ we denote by $G_J\in \matrec{k}$ the submatrix of $G$
given by $G_J = \big( G_{ij}\big)_{i,j\in J}\,$. Similarly, if $x \in \cene$, 
we write $x\geqp 0$ if $x \in \R_+^m$ and $x_J = (x_j)_{j\in J}  \in \C^k$.
\EOE
\end{fed}

\pausa
From the previous definitions and Eq. \eqref{FPA}, 
 we get the fundamental equality: $I(\FS)=I_2(A_{\cW_w})$. Now it is clear why the results of the Appendix can be useful for computing $I(\cW_w)$.

\begin{rem}\label{rem clari}
Let  $\mathbf a\in \RR^m$. Then
\begin{equation}\label{eq obs potencial}
\FP(\mathbf a\cdot \FS)=\suml_{i,j=1}^m a_i^2\, a_j^2\,
w_i^2w_j^2\tr(P_iP_j)=\big\langle (B\circ \mathbf a\, \mathbf a^*) \,
\mathbf a, \mathbf a\big\rangle=\big\langle B (\mathbf a \circ  \mathbf a^*)
,(\mathbf a \circ  \mathbf a^*)\big\rangle
\end{equation} 
where, as before,  $B = B_{\FS} \in \matmreal^+$. 
Moreover, by Eq. \eqref{nopos} in Proposition \ref{variosIB}, 
\begin{equation}\label{clarividencia}
I(\FS)=\min_{\|\pa\|=1}\|A_{\FS}\circ \pa\,\pa^*\|_{_2}
^2=\min_{\|\pa\|=1}\|B\circ \pa\,\pa^*\| = I_{sp}(B) \ .
\end{equation}
This identity is useful because $0\leqp B \in \matmreal^+$ and its spectral index is easier to compute. 
Indeed, observe that $\tr \, P_{W_i} \, P_{W_j} \ge 0$ for every $i,j \in \IM \, $, 
since each $P_{W_i} \in \matpos$.  \EOE
\end{rem}

\subsection{Critical Points and local minimizers.} 
In what follows, we shall use all the assumptions and notations of the previous subsection, but we need the following extra notations: 
\ben
\item Given $\ba\in \RR^m$, we write $z = \ba \circ \pa = (a_1^2 \, , \dots , a_m^2)$
and $J = \sop \{\pa \} = \{ j \in \In : a_j \neq 0\}$.
\item Let $L:\RR^m \rightarrow  \R^+$ be given
by $L(\pa )=\FP(\mathbf a\cdot \FS)=\|A_{\FS}\circ \ba
\ba^*\|_2^2=\pint{B_{\FS}z}{ z}$.
\item We consider the affine manifold $\Delta_0 = \{x\in \RR^m: \tr x=0\}$ and the compact convex 
simplex $\Delta = \{ x\in \Delta _0\, : \, x\geqp 0\}$. 
\item $ S^{m-1} = \{x\in \RR^m: \|x\|=1\}$ is the unit sphere of $\R^m$. 
\een

\begin{lem} \label{RSPC}
Let $0\leqpi \ba\in S^{m-1} $.  Denote  $B = B_{\FS}\,$. 
The following conditions
are equivalent: \ben
\item $\pa$ is a critical point for the map $L$
restricted to the sphere  $ S^{m-1}$.
\item $B_Jz_J=I(B_J)\uno_J \,$, where $J=\sop \{\ba\}\,$.
\een
\end{lem}
 \begin{proof}
 Observe that $L(\pa) = \api Bz, z\cpi = \sum\limits_{i,j=1}^m b_{ij} \, a_i^2 \, a_j^2 \,$.
 Since $B = B^*= B^T$,  we have that
 $$
0=  \nabla L(\ba) = 4 \ \big(   a_1 \, \sum\limits_{j=1}^m b_{1j} \, a_j ^2 \, , \dots ,  \
 a_m \, \sum\limits_{j=1}^m b_{mj} \, a_j ^2  \big) = 4 \ Bz\circ \ba \ .
 $$
 The tangent space of $ S^{m-1}$ at $\pa$ is $\{\pa\}\orto$. Hence, $\pa$ is a 
 critical point for $ S^{m-1}$
 \sii
 $$
 0 = \api \nabla L(\ba) , y \cpi = 4 \ \api B z\circ \ba , y \cpi  
 \peso{for every} y \in \{\pa\}\orto
 \iff Bz \circ \ba \in \gen{\pa } \ .
 $$
 This is clearly equivalent to the equation $B_Jz_J=\la \uno_J \,$, for some $\la \in \R$. 
 In this case,  since $0\leqp B$ with $0<B_{ii}$ for every  $i \in \In\,$  and $0 \leqpi z$, 
 we can conclude that   $\la >0$.  Moreover,  by Proposition \ref{variosIB} applied 
 to the matrix $B_J\,$, we have that
 $$
 B_J \ z_J=\la \uno_J \ \implies  \ B_J \ \frac{z_J}{\la} = \uno_J \ \implies
 I(B_J)\inv  = \tr \ \frac{z_J}{\la} = \api  \uno_J \, , \ \frac{z_J}{\la} \cpi =
 \la \inv \ \api  \uno  \, , \ z  \cpi =  \la \inv \ .
 $$
 Therefore $\la = I(B_J)$ and $B_Jz_J=I(B_J)\uno_J \,$.
 \end{proof}

\begin{teo}\label{RS los minimos}
Let $0\leqpi \ba\in S^{m-1} $ such that  $B_Jz_J=I(B_J)\uno_J\,$, i.e., 
$\pa$ is a critical point of $L$
restricted to $ S^{m-1}$. Then, the following conditions are
equivalent:
\begin{enumerate}
\item 
$I(B_J)=I_{sp}(B)=I(\FS)$.
\item $\ba$ is a global minimum of $L$ restricted to $ S^{m-1}$.
\item $\ba$ is a local minimum of $L$ restricted to $ S^{m-1}$.
\item $Bz\geqslant I(B_J)\uno$. In other words, that  $(Bz)_j\geq I(B_J)$ for every $j\notin J$.

\end{enumerate}
\end{teo}
\begin{proof}
Denote $A = A_{\FS}\,$. Recall that $I_2(A)^2=I_{sp}(B)$, by Eq. \eqref{clarividencia}. 
By Lemma \ref{RSPC},
$$
\|A\circ \ba\ba^*\|_{_2}^2 \ = \ \pint{Bz}{z} \  = \ I(B_J)\,
\pint{\uno }{z} =\ I(B_J) \ .
$$
This gives the equivalence 1 $\leftrightarrow$ 2. Observe that, if
${\bf b}\in S^{m-1}$, then $w = {\bf b} \circ {\bf b} \in \Delta$.
For each  $w\in \Delta$, consider the line $\gamma_w : [0,1] \to
\Delta$ joining $z$ and $w$, given by the formula $\gamma_w
(t)=(1-t)z+tw, $  for every $t\in [0,1]$. Consider the map $\rho_w :
[0,1] \to \R^+$ given by
\begin{align}\label{RS RHO}
\rho_w (t)&=\pint{B\gamma_w(t)}{\gamma_w(t)}=(1-t)^2\pint{Bz}{z} +
t^2 \pint{Bw}{w} +2t(1-t)\pint{Bz}{w} \ ,
\end{align}
for every $ t \in [0,1]$. Since  $\pint{Bz}{z}  = I(B_J) $, the
derivative of $\rho_w$ evaluated at zero is
$$ 
\dot{\rho}_w (0)=-2\pint{Bz}{z}+2\pint{Bz}{w} 
        =2 \Big(\pint{Bz}{w}-I(B_J)\, \Big) \ .
$$ 
On the other hand, for every $t \in \R$, 
\beq\label{RS segu}
\ddot{\rho}_{w}(t)= 2 \pint{Bz}{z} + 2 \pint{Bw}{w} -4  \pint{Bz}{w}
= 2\pint{B(z-w)}{z-w} \ge 0\ . \eeq Since $\rho_w$ is a second
degree polynomial, its leading coefficient is $\frac12 \
\ddot{\rho}_{w}(t) \ge 0$. Suppose now that $Bz\geqslant
I(B_J)\uno$. Using that $w \in \Delta$, we get that
$$
\pint{Bz}{w} \geq I(B_J) \pint{\uno}{w} = I(B_J) \ \implies  \
\dot{\rho}_w (0)\geq 0 \implies \dot{\rho}_w (t)\geq 0  \peso{for every} t\ge 0 \ .
$$
Therefore $\rho_w(1)\geq \rho_w(0)$. In other
words, we have proved  that $\pint{Bw}{w}\geq \pint{Bz}{z}$ for
every $w\in \Delta$. This implies that $\pa$ is a global  minimum of
$L$ restricted to $ S^{m-1}$.

\pausa Suppose now that $\pa \in S^{m-1}$, $z = \pa \circ \pa$, and that
there exists $k\in \IN{m}$ such that  $(Bz)_k<I(B_J)$. Observe that
$\pa \neq e_k\,$, because $(Be_k)_k =  b_{kk} = I(B_{\{k\}} ) $.

\pausa Let $w=e_k\in \Delta$, and consider the curves $\gamma_w$ and
$\rho_w$ defined before.
By the previous computations, we have that  $\dot{\rho}_w(0)=2(\,
(Bz)_k-I(B_J) \, )<0$. Therefore, for every $t> 0$ small enough, we
have that
$\gamma_w (t) \in \Delta$ and
$\pint{B\gamma_w(t)}{\gamma_w(t)}<\pint{Bz}{z}$. Taking the vectors
$\pa(t) = $ sgn$(\pa ) \, 
\gamma_w (t)  \rai   \in S^{m-1}$, we 
conclude that $\pa$ fails to be a local minimum.
\end{proof}

\begin{rem}\rm
Given $\pa \in S^{m-1}$, then $z = \pa \circ \pa \in \Delta$ and
$L(\pa ) = \api Bz, z\cpi$. Therefore
$$
I(\FS ) = I_2 (A ) = \min _{\pa \in S^{m-1}} \, \|A\circ \ba
\ba^*\|_2^2 =
\min _{0\leqpp \pa \in \, S^{m-1} 
}  \, \|A\circ \ba \ba^*\|_2^2
= \min_{z\in \Delta } \, \api Bz, z\cpi \ ,
$$
since every $z\in \Delta $ produces a unit vector $0\leqpi \pa = z\rai \in
S^{m-1}$.  Then in order to get the unit vectors
$\pa$ which attain this minimum, it suffices to characterize the
sets
$$
\cS (\FS )=\argmin_{z\in \Delta} \{ \pint{Bz}{z}\} \peso{and}
J(\FS ) = \big\{ J \inc \IN{m} \, : J = \sop \{z \}  \ \
\mbox{for some} \ \ z\in \cS(\FS) \big\} \ .
$$
If $\IM \notin J(\FS)$, it is possible to obtain minimizers 
$\pa\cdot \FS$ which are not fusion frames, because $S_{\pa\cdot \FS}\notin \matinv^+$
(see Example \ref{coco}). Still, 
if $I(\FS )<\frac {\sqrt{1+n}}{n}$ then $S_{\pa\cdot \FS}\in \matinv^+$ for any minimizer $\ba$, 
since in such case Proposition \ref{dist a la ident} implies that $\|I-n\,S_{\pa\cdot \FS}\|<1$.
Otherwise, the characterization of the set $J(\FS )$ is useful in order to 
discern if there are optimal sequences of weights 
$\pa$ such that $\pa \cdot \FS$ remains being a FF. 
Item 4 of Theorem
\ref{RS los minimos} gives a description of the elements of
$\cS(\FS )$. But its proof gives more information: \EOE
\end{rem}

\begin{cor}\label{RS mins y kerB}
Consider $\FS \in \cS_{m\, , \, n}(\d )$, $A= A_{\FS} $ and $B= B_{\FS}$ as before.
Then \ben
\item
The set  $ \ds \cS(\FS ) =\argmin_{z\in \Delta} \{ \pint{Bz}{z}\}$
is
convex. Moreover, for  any point  $z_0\in \cS(\FS )$, 
$$
\cS(\FS )=\big( z_0+N( B)\, \big) \cap \Delta \ .
$$
\item 
$J(\FS) $ is closed under taking unions, so that
$J_{\FS} = \bigcup J(\FS ) = \bigcup\limits _{z\in \,\cS(\FS )} \sop
\{z \} $ is an element of $J(\FS )$, and there exists $z_1 \in
\cS(\FS)$ with maximal support. \een
\end{cor}

\proof 
\ben
\item
Let $z, w\in \Delta$, and consider the function, defined in Eq.
\eqref{RS RHO}:
$$
\rho_{z, \, w} (t)=\big \api B(\, (1-t)z+t w) \, , \, (1-t)z+t w
\big\cpi  \ , \quad t \in \R \ .
$$
Suppose now that  $w ,z  \in \cS(\FS )$ and $w\neq z$. Using that
$\rho_{z, \, w}$ is of second degree, the equality $\ddot{\rho}_{z, \, w}(t)
= 2\pint{B(z-w)}{z-w} \ge 0$  given by Eq. \eqref{RS segu}, and the
fact $\rho_{z, \, w}(t) \geq 0$ for every $t \in \R$, we can conclude
that
$$
\rho_{z, \, w} \peso{is constant}  \iff  \ \ddot{\rho}_{z, \, w}(t) = 0 \
\iff \  z-w\in N( B) \ .
$$
On the other hand, we have that $\rho_{z, \, w}(1) = \rho_{z, \, w}(0) =
\min \limits_{t \in [0,1]} \, \rho_{z, \, w}(t)$. This implies that the
map $\rho_{z, \, w}$ is constant, so that $\gamma(t)\in \cS(\FS)$ for
every $t\in [0,1]$, and $z-w\in N(B)$. The proof of the fact that 
$( z_0+N(B) \, \big) \cap \Delta \inc \cS(\FS )$ for every $z_0 \in \cS(\FS)$ is similar.
\item
Let $z$ and $w$ in $\cS(\FS )$, with supports $J_1$ and $J_2$
respectively. Then, since the entire line 
$tz+(1-t)w\in \cS(\FS)$
($t\in [0,1]$), if we take $u = tz+(1-t)w $ for any $t \in (0,1)$,
it is easy to see that $u\in \cS(\FS )$ and $\sop \{u \}=J_1\cup J_2\ $.
Since  $J(\FS )$ is finite, also the set
$$
J_{\FS}=\bigcup_{z\in \,\cS(\FS )} \sop \{ z \} \ \in  J(\FS) \ .
$$
 Hence $J_{\FS}$ is the support of some $z_1\in \cS(\FS )$. \QED
\een

\begin{cor} \label{RS bu1}\rm
Let $B= B_{\FS}$ and $A = A_{\FS}$ as before. 
Assume that there exists $v\in \RR^m$ such that $v\geqp 0$ and $B\,
v=\uno$. Denote  $\pa = (\tr v ) \mrai \ (v_{_1}\rai \, , \dots ,v_{_m}\rai )$. 
Then  \beq\label{RS da con a} 
 \|\pa\|=1 \peso{and} (\tr v ) \inv  =\FP(\ba\cdot \FS) =
I_2(A)^2 = I(\FS)\ . \eeq
\end{cor}
\proof The fact that $v\geqp 0$ and $B\, v=\uno$ implies, by
Propositions \ref{variosIB} and \ref{3.3}, that
$$
I_2(A)^2 = I_{sp}(B) = I(B) = I(B_J) = (\tr v ) \inv  \ ,
$$
where $J = \sop \{v \} =  \sop \{\pa \}$. 
Since $z = \frac{v}{\tr v} \in \Delta$, then $\pa = z\rai \in S^{m-1}$ 
and $Bz = I(B_J) \uno$. Hence, by Theorem
\ref{RS los minimos}, we have that $z \in \cS(\FS )$ and $\pa \in
S^{m-1}$ satisfies Eq. \eqref{RS da con a}. \QED

\begin{rem} The results of this section seems to be unknown still for the case of 
vector frames. In this case our restrictions translate to the following: 
Let $\F = (f_i)\subim $ be a frame for $\hil$ such that each $\|f_i\| = 1$ 
(i.e., $d_i = 1 \implies w_i = d_i\mrai = 1$). For $\pa \in \R^m$, we consider the 
sequence $\pa \cdot \F = (a_i \, f_i)\subim \,$, and we define $I(\F ) = 
\min\limits_{\|\pa \|=1 } \fp (\pa \cdot \F )$. Then, all the results 
of the section remain true if one consider the matrices 
$$
A_\F = \Big( \, |\api f_j \, , \, f_i\cpi | \, \Big) _{i, j \in \IM } \in \matm _{sa} 
\peso{and} 
B_\F = \Big( \, |\api f_j \, , \, f_i\cpi |^2  \, \Big) _{i, j \in \IM } \in \matm ^+ \ .
$$
Some proofs are slightly easier in this case, because $I(\F ) =  I_2(A_\F ) = I_2(G_\F )$, 
where $G_\F$ is the Gramm matrix of $\F$: 
$G_\F = \big( \, \api f_j \, , \, f_i\cpi  \, \big) _{i, j \in \IM } \in \matm^+$. Observe that the 
diagonal entries of the three matrices involved are equal to $1$. \EOE
\end{rem}

\begin{exa}\label{coco}
Let $B = \frac 14 \bm{ccc} 4&1 & 3 \\ 1 &4 & 2\\3 &2& 4 \em$. Since $A = (B_{ij}\rai)_{ij\in \IN{3}} \in \mathcal{G}\textit{l}\,(3)^+$, we deduce that $A$ 
 is the Gramm
matrix of a Riesz basis
$\cF$ of $\C^3$. Let $v = (\, \frac 45 \ , \  \frac 45 \ , 0) \geqp 0$. Observe that 
$$
\barr{rl}
B v & = \uno_3 \ \implies \peso{(by Corollary \ref{RS bu1})} z = 
(\tr v)\inv \, v = \Big(\ \frac 12 \ , \frac 12\  , 0\ \Big) \in 
\cS(\cF ) \ . \earr 
$$
Since $N( B) = \trivial$, Corollary \ref{RS mins y kerB} assures that
$\cS(\cF ) = \{ z\}$, and $J_0 = \{1,2\}$ is the maximal support for
$\cS(\cF )$. Taking $\pa = z\rai$, we have that $\pa \cdot \cF$ is
the unique scaled sequence of $\cF$ with minimal Frame Potential, but
it fails to be a frame for $\C^3$, because it has just two non zero
elements. \EOE
\end{exa}

\begin{exa}
It can be proved that every $G \in \cM_3(\C )^+ $ such that rk $G = 2$ and $G_{ii} = 1 $ 
for every $i \in \IN{3}\,$ (considered as the Gramm
matrix of a frame $\cF$ for $\C^2$ with three unitary elements),
satisfies that the minimizers $\pa \cdot \cF$ of the  BF-potential
are  frames for $\C^2$.

\pausa Indeed, given $z\in \cS(\cF )$, it is easy to see that $J =
\sop \{ z\}$ has more than one element (otherwise $z = e_i$ for some
$i \in \IN{3}\,$). If $J = \IN{3}\,$ there is nothing to prove. Assume
that
$\sop \{z\} = J$ with $|J| = 2$. 
If rk $G_{J} < 2$ we must have 
$B_{J} = \bm{cc}1&1\\1&1\em$. 
In this case, $I(B_J) = I_{sp} (B) = 1$. But the unique  matrix $B
\in \cM_3(\C )^+ $ which satisfies that $0 \leqp B  $, $I_{sp} (B) =
1$ and  $B_{ii} = 1$ for every $i \in \IN{3}\,$ is $B = \uno \cdot
\uno ^*$. Indeed, if some $B_{ij}<1$, then
$I_{sp}(B) \le I_{sp} \bm{cc} 1& B_{ij}\\B_{ij}&1\em =
\frac{1+B_{ij}}{2} <1$. Finally, since $1 = $ rk $ \uno \cdot \uno^* =$ rk $B
=$ rk $G\circ \overline{G}   \ge $ rk $G= 2$, we have a
contradiction.
 \EOE
\end{exa}

\begin{rem} Let $\cW=\{W_i\}_{i \in \IM}\, $ be a generating set of subspaces. 
Given a partition $\{J_k\}_{k\in \IN{p}}$  of the set $\IM\,$, we say that  
the sequence $\{\cW_k\}_{k\in \IN{p}}$ of $\cW$ given by 
$\cW _k = (W_i )_{i \in J_k}\,$ is a {\it partition in orthogonal components} 
of $\cW$ (briefly POC) 
if $W_i \perp W_j$ for every pair $i \in J_k\,$, $j \in \IM\setminus J_k\,$.

\pausa
Note that by definition the trivial partition given by $J_1 =\IM$  produces  a
POC of $\cW$. 
If $\{\cW_k\}_{k\in \IN{p}}$ is  a POC  of $\cW$, we say that it is  {\it maximal} 
 if the only POC of each $\cW_k$ is the trivial one. It is clear
that there always exits such a  maximal POC for $\cW$. 

 \pausa
 Let $\{\cW_k\}_{k\in \IN{p}}$ be a maximal POC of 
 $\cW $ with $|J_k|=m_k$ for $1\leq k\leq p$. Let $\mathbf a_k\in
\RR^{m_k}$ be such that $\|a_k\|=1$ and $I(\cW_k)=
\FP(\mathbf a_k\cdot \cW_k)$ for each $1\leq k\leq p$. Then, there exists
$\gamma=(\gamma_k)_{k\in \IN{p}}\in \RR_{> 0}^p$ with $\|\gamma\|=1$ and
such that 
$$ 
I(\cW)=\suml_{k=1}^p \FP(\gamma_k \,\mathbf a_k\cdot \cW_k) \ .
$$
Conversely, if $\mathbf a=(\mathbf a_1,\ldots,\mathbf a_p)$ with 
$0\leqpi \mathbf a_k\in\RR^{m_k}$ and $\|\mathbf a\|=1$
is such that $I(\mathcal W)=\FP(\mathbf a\cdot \cW)$ then 
$\mathbf a_k\neq 0$ and $I(\cW_k)=\FP(\|\mathbf a_k\|^{-1} \,
\mathbf a_k\cdot\cW_k )$, for $1\leq k\leq p$. 
Hence, we can restrict our study of the 
optimal weight of sequence of subspaces to each of the components of the maximal partition. 
This in turn implies that we can reduce the problem of describing the 
optimal weights to the case where the matrix $B$ (which has non-negative 
entries and is positive semi-definite) is irreducible i.e., none of its 
symmetric permutations can be written as the direct sum of two matrices. 
This last property is relevant in the theory of matrices with non-negative entries.
\end{rem}

\appendix{
\section{Hadamard products and  indexes.}

In this section we recall some definitions and results from \cite{[CS]} which are 
closely related with the problems of Section 4. The exposition is done with 
some detail for several reasons: a) Most results we state are explicitly used in the previous section. 
b)  The formulation of these results given in \cite{[CS]} is 
quite technical and intricate, so we intend here to give a clarified version. 
c) Although some results in the appendix are not directly applied, 
they are included since they give effective criteria for computing the indexes and the 
vectors that realize them. 
This is relevant since we have identified
these objects as the optimal weights and the minimal potential 
for fusion frames.

\subsection{Basic definitions and properties}
We begin with an extended version of  Definition \ref{IND1} 
\begin{fed}\label{IND2}\rm
Let $G\in \cM_m(\C )^+$.
\begin{enumerate}
\item The {\bf minimal}-Hadamard index of $G$ is the number 
\[
I(G)=\max\{ \lambda \geq 0\,  : \, G\circ B\geq \lambda B \peso{for every}  B \in \cM_m(\C )^+\} \ .
\]
\item Given an u.i.n  $N$ in $\cM_m(\C )$, the $N$-Hadamard index of $G$ is
\begin{align*}
I_N(G)&= \max\ \big\{\lambda \geq 0 \, : \,  N(G\circ B)\geq \lambda N(B)
\peso{for every} B\in \cM_m(\C )^+ \big\} \\ 
      &= \min \ \big\{ N(G\circ B) \, : \, B\in \cM_m(\C )^+ \text { and } N(B)=1\big\} \ .
\end{align*}
\end{enumerate}
The index of $G$ associated with the spectral norm  $\|\cdot \|=\|\cdot \|_{sp}$
is denoted by  $I_{sp}(G)$, and  the one associated with the 
Frobenius norm $\|\cdot\|_2$ is denoted by  $I_2(G)$. \EOE
\end{fed}

\begin{pro}\label{variosIB} \rm
Let $G\in \cM_m(\C )^+$, $\mathds{1}=(1,1,\ldots,1) \in \C ^m$ and  $\E = \uno \cdot \uno^T$.
\ben
\item $I(G) \neq 0$ \sii \ $\mathds{1} \in R(G)$. If there exists
$y\in \cene$ such that $Gy=\mathds{1}$, then
\beq\label{2.5}
\barr{rl}
 I(G)^{-1}& =\suml_{i \in \IM}\,  y_i =
\api y , \mathds{1} \cpi  = \rho (G^\dagger \E )
= \min \ \{ \ \api Gz,z\cpi  \ : \  \suml_{i \in \IM} z_i = 1 \ \} \ .
\earr
\eeq
If $G>0$, then also 
$
 I(G) = \Big(\sum\limits_{i,j=1}^n (G^{-1})_{ij}\Big)^{-1} = \ds \frac{\det G }{\det(G+\E)-\det G} \ $ .
\item $I(G) \le I_N(G)$ for every u.i.n. $N$.
\item If $J \inc \IN{m}\,$, then $I(G_J) \ge I(G)$ and $I_N(G_J) \ge I_N(G)$. 
\item If $D=\diag{d} \in \cM_m(\C )^+$ is {\bf diagonal}, then $I_N(D) = N'(d\inv ) \inv $. In particular, 
$$
I(D) = I_{sp}(D) = \Big(\sum _{i \in \IM}\,  d_i\inv \Big)\inv \peso {and}
I_2(D) = \Big(\sum _{i \in \IM}\,  d_i ^{-2} \Big) \mrai \ .
$$
\item Both indexes $I_2 $ e $I_{sp}$ are attained by matrices $B \in \cM_m(\C )^+$ of rank one. 
This means that
$$
I_2(G) = \min_{\|x\| = 1} \  \|G\circ xx^* \|_{_2} \peso {and}
I_{sp}(G) = \min_{\|y\| = 1} \  \|G\circ yy^* \| \ .
$$
Moreover, the minima are also attained at vectors  $x\geqp 0$ (or $y\geqp 0$).
\item Let $B = G\circ \overline{G} \in  \cM_m(\R )^+$. Then 
$I_2\,(G)\ =\ I_{sp}\, (B) \rai \ = \  I_{sp}\, (G\circ \overline{G}  ) \rai \ .$
\item Moreover, if $0 \leqp B \in \matmreal^+$ and $A\in \matreal_{sa}$ is given by $A_{ij}=B_{ij}^{1/2}$ for
$1\leq i,\,j\leq m$ then, even if $A \notin \matmreal^+$, the index $I_2(A)$ of 
Definition \ref{IND1}	still satisfies  
\beq\label{nopos}
I_{2}(A) = \min\limits _{\|x\|=1}\|A\circ x\,x^*\|_{_2} =\min\limits _{\|x\|=1}\|B\circ x\,x^*\|\rai \ .
\eeq
\item It holds that
$
I_{sp}( G) = \inf  \ \{ \ I_{sp}( D) : G\le D \ \hbox{ and $D$ is diagonal }\}\, $. Therefore 
\beq
\barr{rl} 
I_{2} ( G)& = \inf \ \{ \  \big(\, \suml _1^n d_{ii}^{-2}\, \big)\mrai : \ 0< D \ \, \mbox{ \rm
is diagonal and } \, G \circ {\overline G} \le D^2 \ \} \ . \QEDP \earr
\eeq
\een
\end{pro}
\begin{pro}\label{3.3}\rm
Let $G\in \cM_m(\R )^+$ such that $0\leqp G$. Then 
$I_{sp}(G)=I(G)\neq 0 \iff $
there exists $u\geqp 0$  such that   $Gu=\mathds{1}$. \QED
\end{pro}

\begin{pro}\label{el x} \rm
Let $G\in \cM_m(\C )^+$. Denote by $P=G\circ {\overline G}$. If $x \in \bsim$, then
$$
\|G \circ xx^*\|_{_2}^2 = \suml_{i,j} |G_{ij}|^2 |x_i|^2 |x_j|^2
= \api P( x\circ x) , x\circ x\cpi =
\ \api  (P\circ xx^*) x ,  x\cpi  \  \le \|P\circ xx^*\| \ .
$$
Take $\ x\geqp 0$ such that
$\| x\|=1$ and $\| G\circ xx^*\|_2 = I_{2} ( G)$.
Then 
$$
\peso{$(G \circ \overline{G} \circ xx^*) x = I(P_J) x $ , \ \ \
where \ \ \  $J = \{ i \in \IN{m}\  : \  x_i \ne 0 \}$ .}
$$
In this case, they hold that
\ben
\item  The vector $u = I(P_J) \inv (x_J \circ x_J) \in \C^J$ has strict
positive entries and $P_J u = \uno_J\, $.
\item $I_{sp} (P) = I_{sp} (P_J) = I(P_J)$.
\item
$I_{sp} (  P) = \|P\circ xx^*\| $. \QED
\een
\end{pro}
}

\fontsize {8}{10}\selectfont

\noi {\bf Pedro Massey, Mariano Ruiz and Demetrio Stojanoff}

\noi Depto. de Matem\'atica, FCE-UNLP,  La Plata, Argentina
and IAM-CONICET  

\noi e-mail: massey@mate.unlp.edu.ar , maruiz@mate.unlp.edu.ar  and demetrio@mate.unlp.edu.ar

\end{document}